\documentclass[reqno,10pt]{amsart}

\usepackage{amsmath,amsfonts,amssymb,amsthm}
\usepackage{mathtools}
\usepackage[normalem]{ulem}
\usepackage{fontenc}
\usepackage[latin1]{inputenc}
\usepackage[cmtip,all]{xy}
\usepackage[dvipsnames]{xcolor}
\usepackage[colorlinks=true]{hyperref}
\usepackage[all]{xy}
\usepackage{url}
 \allowdisplaybreaks \numberwithin{equation}{section}

\newtheorem{theorem}{Theorem}[section]

\newtheorem{proposition}[theorem]{Proposition}
\newtheorem{corollary}[theorem]{Corollary}
\newtheorem{lemma}[theorem]{Lemma}

\newtheorem{claim}[theorem]{Claim}

\theoremstyle{definition}
\newtheorem{definition}[theorem]{Definition}

\theoremstyle{remark}
\newtheorem{remark}[theorem]{Remark}

\DeclareMathOperator{\Sym}{Sym}
\DeclareMathOperator{\Proj}{Proj}
\DeclareMathOperator{\Pic}{Pic}
\DeclareMathOperator{\End}{End}

\DeclareMathOperator{\seg}{\mathbf{s}}

\def\bH{{\mathbb H}}
\def\bP{{\mathbb P}}
\def\bQ{{\mathbb Q}}
\def\bZ{{\mathbb Z}}

\newcommand{\vac}{{\bf{1}}}

\def\cI{{\mathcal I}}

\def\cO{{\mathcal O}}

\def\leq{{\leqslant}}
\def\geq{{\geqslant}}

\DeclareMathOperator{\rk}{rk}

\begin{document}

%\setreviewsoff
%\setreviewson

\title[Some families of big and stable bundles]{Some families of big and stable bundles on $K3$ surfaces and on their Hilbert schemes of points}

\author{Gilberto Bini}
\address{Gilberto Bini, Dipartimento  di Matematica ed Informatica, Universit\`{a} degli Studi di Palermo, Via Archirafi 34, 90133 Palermo -- Italy}
\email{gilberto.bini@unipa.it}

\author{Samuel Boissi\`ere}
\address{Samuel Boissi\`ere, Laboratoire de Math\'ematiques et Applications, Universit\'e de Poitiers, 
	UMR 7348 du CNRS, Site du Futuroscope, TSA 61125, 11 bd Marie et Pierre Curie, 86073 Poitiers Cedex 9, France}
\email{samuel.boissiere@univ-poitiers.fr}

\author{Flaminio Flamini}
\address{Flaminio Flamini, Dipartimento  di Matematica, Universit\`{a} degli Studi di Roma ``Tor Vergata", Viale della Ricerca Scientifica 1, 00133 Roma -- Italy}
\email{flamini@mat.uniroma2.it}

\thanks{This collaboration has benefitted from funding of the MIUR Excellence Department Project awarded to the Department of Mathematics,
University of Rome Tor Vergata (CUP: E83-C18000100006). The first and the third authors are members of INdAM--GNSAGA. The authors wish to 
warmly thank Sandra Di Rocco, Andreas L. Knutsen, Angelo F. Lopez, Dragos Oprea, Antonio Rapagnetta and Alessandra Sarti for useful discussions and remarks, and the anonymous referee for his/her very careful reading and helpful comments.}

\begin{abstract} Here we investigate meaningful families of vector bundles on a very general polarized $K3$ surface $(X,H)$ and on the corresponding {\em Hyper--K\"ahler variety} given by the Hilbert scheme of points $X^{[k]}:= {\rm Hilb}^k(X)$, for any integer $k \geqslant 2$. In particular, we prove results concerning bigness and stability of such bundles. 
	First, we give conditions on integers $n$ such that the twist of the tangent bundle of $X$ by the line bundle $nH$ is big and stable on~$X$; we then prove a similar result for a natural twist of the tangent bundle of $X^{[k]}$. Next, we prove global generation, bigness and stability results for 
	{\em tautological bundles} on $X^{[k]}$ arising either from line bundles or from {\em Mukai-Lazarsfeld} bundles, as well as from {\em Ulrich} bundles on $X$, using a careful analysis on Segre classes and numerical computations for $k = 2, 3$.
\end{abstract}

\maketitle

%%%%%%%%%%%%%%%%%%%%%%%%%%%%%%%%%%%%%%%%%%
%%%%%%%%%%%%%%%%%%%%%%%%%%%%%%%%%%%%%%%%%%
%%%%%%%%%%%%%%%%%%%%%%%%%%%%%%%%%%%%%%%%%%

\section{Introduction}\label{intro} 

Let $M$ be an $n$-dimensional smooth projective variety over the field of complex numbers. A line bundle $L$ on $M$ is big if the Iitaka dimension $\kappa(M,L)$ is maximal. If $L$ is in particular nef, the bigness of $L$ can be deduced from a numerical criterion, that is, $L$ is big if and only if the top intersection of $c_1(L)$ is a positive integer. If we take into account vector bundles $E$ of rank greater than or equal to $2$, there are various notions of bigness: see, for instance, \cite{MU} for a comprehensive survey. Here we focus on $L$-bigness, whose definition is modeled on those of ampleness and nefness for vector bundles. More specifically, $E$ is $L$-big if and only if the tautological bundle ${\mathcal O}_{{\mathbb P}(E)}(1)$ is a big line bundle on the projective bundle $\pi: {\mathbb P}(E) \to M$ of one-dimensional quotients of $E$. Henceforth, for the sake of notation, by big we mean $L$-big.

Notably, the numerical criterion for bigness of nef line bundles induces a characterization for bigness of nef vector bundles of higher rank. Proposition \ref{prop:bigness} recalls this numerical criterion which states that a nef rank $r$ vector bundle~$E$ on~$M$ is big if and only if the number $(-1)^n \int_M s_n(E)$ is positive, where $s_n(E)$ is the top Segre class of the vector bundle $E$. In what follows, we will apply this characterization to globally generated vector bundles, which are in fact nef: see Remark~\ref{rem:numerical}. 

In \cite{BF}, we introduced cohomological criteria on algebraic surfaces and fourfolds in order to verify the numerical characterization mentioned before. What's more, we found out examples of big vector bundles (split and unsplit) on Hirzebruch surfaces and investigated the bigness of some families of Mukai-Lazarsfeld bundles on regular fourfolds. These criteria were also applied in \cite{AL} to describe non-big Ulrich bundles on a complex irreducible smooth projective surface. For the sake of completeness, we recall that $E$ is an Ulrich bundle on $M$ if $H^i(M, E(-p))$ vanishes for $i \geq 0$ and $1 \leq p \leq \dim(M)$.

In the present paper, we investigate bigness of vector bundles on other families of varieties, namely {\em Hyper--K\"ahler varieties}. Previous results on base loci of big and nef line bundles on them were obtained in~\cite{ur}. The Beauville-Bogomolov Theorem (see \cite[Theorem~1]{BB}) states that, up to a finite cover, any compact K\"ahler manifold with trivial first Chern class (in the rational cohomology) can be decomposed as a product of complex tori, (strict) Calabi-Yau varieties and Hyper--K\"ahler varieties. By definition, the latter have even complex dimension. The first examples are thus given by $K3$ surfaces $X$. For the purposes of what follows, we shall focus on very general polarized $K3$ surfaces $(X, H)$ where $H$ is an ample divisor on $X$ such that $H^2=2g-2 \geq 2$. By general results (see for instance \cite[Theorem VIII 7.3 on page 366]{BPV}) there exists a smooth, irreducible $19$-dimensional moduli space ${\mathfrak F}_g$ which parametrizes (isomorphism classes of) smooth, primitively polarized pairs $(X,H)$ of genus $g$. A very general point of~${\mathfrak F}_g$ corresponds to a {\em very general polarized} $K3$ surface $(X,H)$.

Any rank $r$ vector bundle $E$ on $X$ defines a {\em Mukai vector} $v:=v(E)$: see \S\ref{bundlesK3} for the definition. Correspondingly, one denotes by $M_H(v)^s$ the moduli space of $\mu_H$-stable vector bundles on $(X,H)$ associated with the Mukai vector~$v$, where $\mu_H(E):=\frac{\int_X c_1(E) \cdot H}{r}$is the $H$-slope of $E$. Moreover, let us consider the vector bundle $T_X(n):= T_X \otimes H^{\otimes n}$ on $(X,H)$ and denote by $v_{g,n}$ the associated Mukai vector, where $g$ is the genus of $(X,H)$. Then Theorem~\ref{thm:tangK3} lists pairs of possible integers $(g,n)$ such that $T_X(n)$ is big. Moreover, for any such pair the moduli space $M_H(v_{(g,n)})^s$ is a smooth, quasi-projective variety of dimension $90$ whose general element parametrizes a rank $2$ vector bundle with Mukai vector $v_{(g,n)}$. The proof follows from explicit calculations by the Hirzebruch--Riemann--Roch Theorem and Proposition \ref{prop:Laz}. The latter gives a useful criterion for the bigness of a vector bundle; see \cite[ Example 6.1.22]{Laz2} for further details. As for the proof of stability of $T_X(n)$ in the cases above, we apply a series of results which are recalled in Proposition \ref{prop:modspace} (see for instance~\cite{Hu}).

Denote by $X^{[k]}:=Hilb^k(X)$ the Hilbert scheme of zero-dimensional subschemes of length $k$ on a polarized $K3$ surface $(X,H)$, see e.g. \cite{Bos, Dan1, Dan2, Kr, Sca1, Sca2, Sta}. This turns out to be a Hyper--K\"ahler variety. Similarly to the case of polarized K3 surfaces, we prove some bigness and stability results for the rank $2k$ tangent bundle up to a twist. More precisely, set $Y=X^{[k]}$ and consider the tangent bundle $T_Y$ on $Y$. The polarization $H$ on the $K3$ surface~$X$ gives a big and nef line bundle $H_Y$ as in \eqref{eq:Dk}. Then Theorem \ref{thm:tangHilbK3} lists pairs of positive integers $(g,n)$ such that $T_Y \otimes H_Y^{\otimes n}$ is a $\mu_{H_Y}$-stable, where $g$ is the genus of $X$. 

Besides the tangent bundle, we also focus on other families of vector bundles, first on very general polarized $K3$ surfaces and after that on the Hilbert scheme of points on them. Henceforth, assume the genus of $X$ is greater than or equal to $3$. For a primitively polarized $K3$ surface, pick a general curve $C \in |H|$ and a complete linear series $|A|=g^{r-1}_d$ on $C$, with suitable properties of global generation on $A$ and integrality assumptions on any member of~$|H|$; see \S\ref{oddrank} for precise statements. Lazarsfeld defines a rank $r$ vector bundle $E:=E_{C,A}$ on $X$, which encodes several properties of Brill-Noether and Petri theory of the scheme parametrizing special linear series on $C$. The vector bundle $E$ is usually called the {\em Mukai-Lazarsfeld vector bundle} associated with the pair $(C,A)$. In this setting, Theorem \ref{prop:oddvect} proves that for any triple of positive integers $(g,r-1,d)$ such that $d < 2g-2$ and the Brill-Noether number $\rho(g,r-1,d) \geq 0$, there exists a vector bundle $E$ which is globally generated and $\mu_H$-stable on $X$. Moreover, if $\rho(g,r-1,d)=0$ the moduli space $M_H(v)^s$ consists of a single reduced point, which yields an (isomorphism class) of big bundles. If $\rho(g,r-1,d)$ is positive, the general bundle in the moduli space $M_H(v)^s$ is globally generated and big. Here $v$ is the Mukai vector $v=(r,H, g-1-d+r)$.

Another family of examples comes from Ulrich bundles on very general polarized surfaces of genus $g=h+1$, which are dealt with in \S\ref{evenrank}. There, Theorem \ref{thm:ultrich} proves that for any positive integer $a \geq 1$ there exists an $(8a^2+2a^2h+2)$-dimensional family of $\mu_H$-stable Ulrich bundles $E$ on $X$ with Mukai vector $v=(2a, 3aH, 2a(h-1))$. More geometrically, the general point of the moduli space $M_H(v)^s$ corresponds to a $\mu_H$-stable Ulrich bundle of rank~$2a$, which is also globally generated and big.

If $E$ is a rank $r$ vector bundle, the Hilbert scheme $X^{[k]}$ carries a natural rank $rk$ vector bundle $E^{[k]}$, which is known as the {\em tautological bundle} associated with $E$: see \S\ref{bundlesHK}. Then it is natural to consider tautological bundles arising from Mukai-Lazarsfeld bundles and Ulrich bundles, with the same notation and assumptions in Theorem~\ref{prop:oddvect} and Theorem~\ref{thm:ultrich}, respectively. In order to analyze their bigness and stability, and to apply the numerical characterization mentioned before, we need to determine if these tautological bundles are globally generated and if they fulfill the numerical criterion in Proposition \ref{prop:bigness}. 

To this extent, we proceed as follows. First, set $L_n=H^{\otimes n}$ where $H$ is the polarization on $X$. As proved by Voisin~\cite{V}, the tautological bundle $L_1^{[k]}$ is generated by global sections for $g > 2k-2$, where $g$ is the genus of the polarized pair $(X,H)$. As for $L_n^{[k]}$, Theorem 5.2 shows that this is a globally generated rank $k$ vector bundle on $X^{[k]}$. The proof is based on the notion of $(k-1)$-very ampleness of $L_n$. In particular, a vector bundle $E$ is $(k-1)$-very ample on $X$ if and only if the tautological bundle $E^{[k]}$ is globally generated on the Hilbert scheme $X^{[k]}$: see Proposition~\ref{prop:speriamo}(i). Proposition~\ref{prop:speriamo}(ii) proves that if $E$ is globally generated and $L_1^{[k]}$ is globally generated, i.e. $L_1$ is $(k-1)$-very ample, then $(E \otimes L_1)^{[k]}$ is globally generated. We will apply this proposition to a Mukai-Lazarsfeld bundle or an Ulrich bundle~$E$. To this end, we suitably adapt arguments used in \cite{BS} for the stronger notion of $k$-jet ampleness of vector bundles. As for the bigness behaviour, it remains to check $\int_{X^{[k]}}s_{2k}((E \otimes L_1)^{[k]})>0$, where $E$ is either a Mukai-Lazarsfeld bundle or an Ulrich bundle on $X$. 

Therefore, we need a formula for the top Segre class of a tautological bundle $F^{[k]}$ for a rank $r \geq 1$ vector bundle~$F$ on~$X$. The total Segre class $s(F^{[k]})$ is computed in Proposition \ref{prop:segre_taut} via recurrence relations, which are based on a short exact sequence connecting the tautological bundle $F^{[k-1]}$ and the tautological bundle $F^{[k]}$ by pulling them back on the incidence scheme $X^{[k-1,k]}$ parametrizing triples $(\xi, x, \xi') \in X^{[k-1]} \times X \times X^{[k]}$ such that $\xi \subset \xi'$ with residual subscheme supported at the point $x$: see \S\ref{segre_tautHilb} for further details. From such relations, we obtain recurrence formulas expressing any $d$-th Segre class $s_d(F^{[k]})$ by induction on $k$ (cf. Formula \eqref{eq:dth_Segre}). As an example, we give an explicit formula for $s_4(F^{[2]})$ and $s_6(F^{[3]})$ in Corollary~\ref{cor:segre_taut}. This requires the formalism of the Heisenberg algebra (cf. e.g.~\cite{Lehn}), which describes the cohomology algebra structure of the Hilbert scheme of points on a surface. Segre integrals are also computed in \cite{MOP}, for any $k\geqslant 2$, as coefficients of suitable generating series ; see also~\cite{MOP_rank1, V} for the rank $1$ case. Our formulas agree with theirs when $k = 2, 3$.

%In Appendix~\label{app:oprea} written by Dragos Oprea, the general statements given %in Theorems~\ref{thm:lbHilb}, \ref{prop:oddvectHilb} \& \ref{prop:evenvectHilb} and %generating series of Segre integrals of~\cite{MOP}  are used to extend the numerical %computations in Corollaries~\ref{cor:lbHilb}, \ref{cor:oddvectHilb} \& %\ref{cor:evenvectHilb} to $k\geq 4$ and to obtain
%similar positivity results on the Hilbert schemes of points of abelian, bielliptic %or Enriques surfaces.

%%%%%%%%%%%%%%%%%%%%%%%%%%%%%%%%%%%%%%%%%%
%%%%%%%%%%%%%%%%%%%%%%%%%%%%%%%%%%%%%%%%%%
%%%%%%%%%%%%%%%%%%%%%%%%%%%%%%%%%%%%%%%%%%

\subsection*{Notation and terminology}

Throughout, we work over the field $\mathbb{C}$ of complex numbers. By \emph{variety} we mean an integral algebraic scheme $Y$. 
We say that a property holds for a \emph{general} (resp. \emph{very general}) point ${y\in Y}$ if it holds on a Zariski open nonempty subset of
 $Y$ (resp. on the complement of the countable union of proper subvarieties of~$Y$). For any vector bundle $E$ on $Y$, we denote by $S^mE$ the $m^{th}$--symmetric power of $E$ and by $\Sym(E)$ the symmetric algebra. We denote by $T_Y$ the tangent bundle of $Y$.

%%%%%%%%%%%%%%%%%%%%%%%%%%%%%%%%%%%%%%%%%%
%%%%%%%%%%%%%%%%%%%%%%%%%%%%%%%%%%%%%%%%%%
%%%%%%%%%%%%%%%%%%%%%%%%%%%%%%%%%%%%%%%%%%

\section{Preliminaries}\label{Pre} 

\subsection{Chern and Segre classes}\label{ChernSegre}

Let $X$ be a smooth, complex projective variety of dimension $n \geq 2$ and  $E$ be a rank $r$ vector bundle on $X$, $r \geq2$. We set
 $\mathbb{P}(E) \coloneqq \Proj (\Sym(E))$ the projective bundle parametrizing $1$-dimensional quotients of the fibres of $E$, we denote by ${\mathcal O}_{\mathbb{P}(E)} (1)$ the {\em tautological line bundle} on $\mathbb{P}(E)$ and by $\mathbb{P}(E) \stackrel{\pi}{\longrightarrow} X$ the canonical projection (cf. e.g. \cite{H}).
We define the $i^{th}$ \emph{Segre class} of $E$ as:
\begin{equation}\label{eq:IID4}
s_i(E)\coloneqq \pi_*\left(c_1({\mathcal O}_{\mathbb{P}(E)} (1))^{r-1+i}\right)\in H^{2i}(X,\bZ).
\end{equation}
The {\em total Segre class} of $E$ is given by $s(E) \coloneqq 1 + s_1(E) + s_2(E) + \cdots \in H^\ast(X,\bZ)$. The \emph{Chern classes} 
$c_i(E)\in H^{2i}(X,\bZ)$ of $E$ are defined as the coefficients of the inverse formal series of 
$s_E (t) \coloneqq \sum_{i=0}^{+\infty} s_i(E) t^i$, i.e. $c_E(t) = s_E(t)^{-1}$ and $c_E (t) \coloneqq \sum_{i=0}^{+\infty} c_i(E) t^i$. We put $c(E)\coloneqq c_E(1)$. In particular, one has (cf.\;e.g.\;\cite[\S\,3.2]{Fu}):
\begin{equation}\label{eq:segre}
c_1(E) = - s_1(E), \; c_2(E) = s_1(E)^2 - s_2(E), \ldots, c_i(E) = - s_1(E) c_{i-1}(E) - s_2(E) c_{i-2}(E) - \cdots - s_i(E),\; \forall \; i \geqslant 3.
\end{equation}
Denoting by $E^{\vee}$ the dual bundle of $E$, for any line bundle $L$ on $X$ one has (cf. \cite[Rem.\;3.2.3\;(a),\;Ex.\;3.2.2,\;Ex.\;3.1.1]{Fu}): 
\begin{align}
c_i(E^{\vee}) &= (-1)^i c_i(E) \;\; {\rm and} \;\; c_i(E \otimes L) = \sum_{j=0}^i {r-j \choose i - j} c_j(E) c_1(L)^{i-j}, 
\label{eq:dualchern}\\
s_i(E^{\vee}) &= (-1)^i s_i(E) \;\; {\rm and} \;\; s_i(E \otimes L) = \sum_{j=0}^i (-1)^{i-j}{r-1+i \choose r-1 + j} s_j(E) c_1(L)^{i-j}.
\label{eq:dualsegre}
\end{align}

%%%%%%%%%%%%%%%%%%%%%%%%%%%%%%%%%%%%%%%%%%
%%%%%%%%%%%%%%%%%%%%%%%%%%%%%%%%%%%%%%%%%%
%%%%%%%%%%%%%%%%%%%%%%%%%%%%%%%%%%%%%%%%%%

\subsection{Positivity of vector bundles}\label{Positivity}

We remind some definitions concerning certain {\em dimension} and {\em positivity} notions related to vector bundles over $X$; for more details, we refer the reader 
to~\cite[\S\,2.2]{BF} and to references therein. These concepts are first defined for line bundles $L$ on $X$ and then for vector bundles $E$ of rank $r \geq 2$  by considering the associated line bundle $\cO_{\bP(E)} (1)$ on  $\bP(E)$.

\subsubsection{Kodaira--Iitaka dimension, bigness and nefness}\label{KodairaItaka}

Let $L$ be a line bundle on $X$; its {\em Kodaira--Iitaka dimension} $k(L)$ is defined as:
\[
k(L) \coloneqq \left\{\begin{array} {cl} 
- \infty & {\rm if} \; h^0(L^{\otimes m}) = 0, \; \forall \; m \in \mathbb{N} \\
{\rm max}_{m \in \mathbb{N}} \; \dim (\varphi_{L^{\otimes m}}(V)), & {\rm otherwise,} 
\end{array}
\right.
\]
where $X \stackrel{\varphi_{L^{\otimes m}}}\dashrightarrow \mathbb{P} (H^0(L^{\otimes m})^{\vee})$ denotes the rational map given by the linear system $|L^{\otimes m}|$. Then $L$ is said to be {\em big} if $k(L) = n = \dim (X)$, and
$L$ is called {\em nef} if $L \cdot C \geqslant 0$ for any effective curve $C \subset X$.

Let now $E$ be any rank $r$ vector bundle on $X$, with $r \geqslant 2$. Similary as above, its {\em Kodaira--Iitaka dimension} $k(E)$ is defined to be $k(E)\coloneqq 
k({\mathcal O}_{\mathbb{P}(E)} (1))$. The vector bundle $E$ is said to be {\em big}  
if ${\mathcal O}_{\mathbb{P}(E)} (1)$ is big on $\mathbb{P}(E)$. We thus have that
$E$ is big if and only if $k(E) = \dim(\bP(E)) = n + r - 1$.
It is said to be {\em nef} if ${\mathcal O}_{\mathbb{P}(E)} (1)$ is a nef line bundle on~$\mathbb{P}(E)$ (cf.\;e.g. \cite[Definition~1.9]{DPS}).  We recall for later use the following result (cf. \cite[Ex.\;6.1.22]{Laz2}):

\begin{proposition}\label{prop:Laz}
Assume that $H^0(X, S^mE) \neq 0$ for some $m \geq 1$. Then for any ample line bundle $A$ on $X$, the vector bundle~$E \otimes A$ is big. 
\end{proposition}

\subsubsection{Numerical dimension}\label{Numdim}

\begin{definition} (cf.\cite[II.E,\,p.\,24]{GG}) \label{def:numdim} Let $L$ be any \emph{nef} line bundle on $X$. The {\em numerical dimension} of $L$ is defined to be 
the largest integer $n(L)$ such that $c_1(L)^{n(L)} \neq 0$.
\end{definition}

By \cite{Dem}  (cf.\;also\;\cite[(II.E.1),\;p.24]{GG}) one has  
$k(L) \leq n(L)$, and equality holds if $n(L) = \dim(X)$.
Let now $E$ be a globally generated vector bundle, of rank $r \geqslant 2$. From Remark \ref{rem:numerical}, $E$ is nef, i.e. ${\mathcal O}_{{\mathbb P}(E)}(1)$ is a nef line bundle on ${\mathbb P}(E)$. Therefore, it makes sense to consider 
its numerical dimension: 

\begin{definition}(\cite[\S\,II.E,\;p.25]{GG}) \label{def:numdim2} Let $E$ be a globally generated vector bundle of rank $r$ on $X$. The {\em numerical dimension} of $E$ is 
$n(E) \coloneqq n({\mathcal O}_{\mathbb{P}(E)} (1))$.
\end{definition}

Since ${\mathcal O}_{\mathbb{P}(E)} (1)$ is very ample on the fibres of the projection $\mathbb{P}(E) \stackrel{\pi}{\longrightarrow} X$, one has
$r-1 \leqslant n(E) \leqslant \dim(\mathbb{P}(E))$.
On the other hand, since $E$ is nef, by \eqref{eq:IID4} and Definition \ref{def:numdim}, observing that the morphism $\pi_\ast$ consists in integrating over the fibers, we have that  
$n(E)$ is the largest integer with $s_{n(E)-r+1}(E) \neq 0$.
We also have
$k(E) \leqslant n(E)$, where the equality holds when $n(E) = \dim(\bP(E)) = n + r-1$.

\subsubsection{A numerical characterization of bigness}\label{numbig}

The following classical result gives a numerical criterion for the bigness of a nef vector bundle. We briefly recall the proof for the reader's convenience.

\begin{proposition}\label{prop:bigness}
Let $E$ be a nef vector bundle on a $n$-dimensional smooth projective variety~$X$. Then $E$ is big if and only if $(-1)^n \int_X s_n(E) > 0$.
\end{proposition}

\begin{proof}
Collecting all the notions recalled above, since $E$ is nef, its  numerical dimension $n(E)$ coincides with the largest integer for which 
$s_{n(E)-r+1} \neq 0$. Then $n(E)-r+1 = n $ is equivalent to $k(E) = n(E) = \dim  (\bP(E))$, meaning that $E$ is big, so the bigness of $E$ is equivalent
to the non-vanishing of its top Segre class.  Applying \cite[Theorem\;2.5]{DPS} with $n = k$, $Y = X$ and $a=(1^n)$, this is equivalent to
the positivity of the Schur polynomial $P_{(1^n)} (c(E)) = s_n (E^{\vee})=(-1)^n s_n(E)$  (cf. \cite[8.3.5]{Laz2}).
\end{proof}

\begin{remark}\label{rem:numerical} If $E$ is globally generated, then $E$ is nef. Indeed, taking ${\mathbb P}(E) \stackrel{\pi}{\rightarrow} X$ the natural projection, global generation of $E$ ensures that $\pi^*E$ is globally generated, hence nef as well.
\end{remark} 

In what follows, we will apply Proposition \ref{prop:bigness} to some globally generated vector bundles.

%%%%%%%%%%%%%%%%%%%%%%%%%%%%%%%%%%%%%%%%%%
%%%%%%%%%%%%%%%%%%%%%%%%%%%%%%%%%%%%%%%%%%
%%%%%%%%%%%%%%%%%%%%%%%%%%%%%%%%%%%%%%%%%%

\section{On some big and stable vector bundles on K3 surfaces}\label{bundlesK3} 

Let $(X,H)$ be a smooth polarized $K3$ surface of genus $g \geqslant 2$, where $H$ is an ample divisor on $X$ such that $H^2 = 2g-2 \geqslant 2$.
From \cite[Def.\,9.1.2, p.169]{Hu}, if $E$ is a rank $r$ vector bundle on $X$, its {\em Mukai vector} is defined as
\begin{equation}\label{eq:Mukaivect}
v(E) \coloneqq (r, c_1(E), \chi(E) - r) = \left(r, c_1(E), \int_X\left( \frac{c_1(E)^2}{2} - c_2(E) \right) + r\right) \in H^0(X, \bZ) \oplus H^2(X, \bZ) \oplus H^4(X,\bZ) \eqqcolon H^*(X, \bZ).
\end{equation}Moreover, for $\alpha = (\alpha_0, \alpha_2, \alpha_4), \beta= (\beta_0, \beta_2, \beta_4) \in H^*(X, \bZ)$, the 
{\em Mukai pairing} is defined as 
\begin{equation}\label{eq:Mukaipair}
\langle \alpha, \beta \rangle \coloneqq (\alpha_2 , \beta_2) - (\alpha_0 , \beta_4) - (\alpha_4 , \beta_0),
\end{equation} 
where $(-, -)$ is the intersection pairing (cf. \cite[Def.\;9.4, p.169]{Hu}).

Any Mukai vector $v:= v(E)$ can be uniquely written as $v = m v_0$, where $m \geqslant 0$ is an integer and $v_0 \in H^*(X, \bZ)$ is indivisible i.e. 
{\em primitive}, equivalently the integer $m$ is maximal (cf.\;\cite[\S\;10.2, p.198]{Hu}). When $m=1$, then $v$ itself is a {\em primitive} Mukai vector. From \cite[\S\,9.3, p.\,175]{Hu}, the {\em $H$-slope} of a rank $r$ vector bundle $E$ on $X$ is defined as 
\[
\mu_H(E) \coloneqq \frac{\int_X c_1(E) \cdot H}{r}
\]
and $E$ is said to be $\mu_H$-{\em stable} (or {\em slope-stable}) if for all subsheaves 
$F \subset E$ with $0 < \rk F <\rk E$ one has $\mu_H(F) < \mu_H(E)$.  Recall that the $\mu_H$-stability is preserved by taking dual bundles and by tensoring with line bundles. 

For a given Mukai vector $v \in H^*(X, \bZ)$, one denotes by $M_H(v)^s$ the {\em moduli space} of $\mu_H$-stable vector bundles 
on $(X,H)$ of given Mukai vector $v$. 
We recall a result of~\cite{Hu} stated for the moduli spaces of Gieseker semistable sheaves; here we consider the open subscheme of slope-stable vector bundles (see also~\cite[Theorem~0.1]{Yo}  or~\cite[Theorem~4.4]{KLS}). 

\begin{proposition}\label{prop:modspace}  Either $M_H(v)^s$ is empty or it is a smooth, quasi-projective scheme of (equi-)dimension $2 + \langle v, v \rangle$ (cf. \cite[\S;10,\; Cor.\,2.1, p.\,196]{Hu}). Moreover: 

\begin{itemize}
\item[(i)] If  $ \langle v, v \rangle = - 2$, then $M_H(v)^s$ is either empty or it consists of one reduced point (cf. \cite[\S\;10,\;Prop.\,3.1, p.\,200]{Hu}).

\item[(ii)] If $v$ is primitive, $H$ is a general polarization, the rank $r \geqslant 1$ and $\langle v, v \rangle \geqslant -2$, then either $M_H(v)^s$ is empty or it is an open dense subset of 
an irreducible symplectic projective manifold $M$ which is deformation equivalent to $Hilb^{2 + \langle v, v \rangle}(X)$. In particular $M_H(v)^s$ is irreducible (cf. \cite[\S\;10,\;Thm.\,3.10, p.\,205]{Hu}). 

\item[(iii)] Let $(X,H)$ be a polarized either 
$K3$ or abelian surface and let $v$ be a Mukai vector of positive rank. If $\langle v,v\rangle > 0$ and if $H$~is general with respect to~$v$,  then $M_H(v)^s$ is either empty or it is an irreducible normal  variety (cf.~\cite[Theorem~0.1]{Yo}) . 
\end{itemize}
\end{proposition}

%%%%%%%%%%%%%%%%%%%%%%%%%%%%%%%%%%%%%%%%%%
%%%%%%%%%%%%%%%%%%%%%%%%%%%%%%%%%%%%%%%%%%
%%%%%%%%%%%%%%%%%%%%%%%%%%%%%%%%%%%%%%%%%%

\subsection{The tangent bundle of a very--general polarized K3}\label{tangent} For any polarized $K3$ surface $(X,H)$
one has (cf. \cite[Example\,9.1.6, p.\,170]{Hu}): 
\begin{equation}\label{eq:tangent}
v(T_X) = (2,0,2-e(X)) = (2, 0, -22) \;\; {\rm and} \;\; \langle v(T_X), v(T_X) \rangle = 88.
\end{equation} 

\vskip 5pt

\noindent
Moreover, the following properties hold:
\begin{align}\label{eq:tangproperties}
(i) & \, h^0(X, T_X) = 0 & \mbox{(cf. \cite[\S\,1.2.4, p.\,13]{Hu})}\\ \nonumber
(ii) & \,h^0(X, S^mT_X) =0, \;\; \forall \;\; m \geq 1 & \mbox{(cf. \cite[Cor.\,9.4.13, p.\,183]{Hu})}\\
(iii) &\,  T_X \, \mbox{is} \, \mu_H-\mbox{stable} & \mbox{(cf. \cite[Prop.\,9.4.5, p.\,180]{Hu})}\nonumber 
\end{align} 

As explained in Introduction, in this paper we are interested in big vector bundles. From \eqref{eq:tangproperties}--(i) and (ii), no $m^{th}$-symmetric power of $T_X$ can be globally generated, for any $m\geqslant 1$. This means that $T_X$ does not satisfy assumptions as in Proposition \ref{prop:Laz} which, therefore, cannot be applied. 

Thus, in what follows, we are concerned in finding sufficient conditions ensuring the existence of a suitable positive integer $n_0$ for which $T_X \otimes H^{\otimes n}$ is big, for any integer $n \geq n_0$. 

To do so, we shall focus on {\em very--general polarized $K3$ surfaces} in the following sense: from \cite[Thm.VIII 7.3 and p.~366]{BPV}, there exists a smooth, irreducible moduli space $\mathfrak{F}_g$ of dimension $19$ which parametrizes 
(isomorphism classes of) smooth, primitively polarized $K3$ surfaces $(X,H)$ of genus $g \geqslant 2$. The pair $(X,H) \in \mathfrak{F}_g$ is called {\em a very general (polarized) K3} when 
(in the sense of the Introduction) $(X,H)$ corresponds to a very general point of~$\mathfrak{F}_g$. Moreover, when $(X,H)$ is very--general, one in particular has 
${\rm Pic}(X) \cong \mathbb{Z}[H]$. In this case, from \eqref{eq:tangproperties}-(iii), \eqref{eq:tangent} and Proposition~\ref{prop:modspace}--(iii) we get: 
\begin{equation}\label{eq:modspacetang}
M_H((2,0,-22))^s \;\; \mbox{is an irreducible, smooth quasi-projective variety of dimension $90$.}
\end{equation} 

To simplify notation, from now on we will moreover identify multiplicative notation of tensor power of line-bundles and additive notation of Cartier divisors, namely $H^{\otimes n}$ will be simply denoted by $nH$. Similarly, we will simply set $T_X(n) \coloneqq T_X \otimes H^{\otimes n}$ for any $n \geq 1$. Taking into account the isomorphism 
$T_X \cong \Omega_X^1$, we can reformulate the results in \cite[\S\,5.2]{Be} as follows:

\begin{proposition}\label{prop:Beauville} Let $(X, H)$ be very--general polarized $K3$ surface of genus $g \geq 2$. Then, one has:
\begin{itemize}
\item[(a)] $h^0(X, T_X(1)) = 0$, for $2 \leq g \leq 9$ or  $g=11$,
\item[(b)] $h^0(X,T_X(1)) = 1$, for $g=10$, 
\item[(c)] $h^1(X,T_X(1)) = 0$, for $g=11$ or $g \geq 13$,
\item[(d)] $h^1(X, T_X(1)) \geq 1$, for $g=12$.
\end{itemize}
\end{proposition}

\begin{remark}\label{gencomp} For any integer $n \geq 0$ one has
$$c_1(T_X(n)) = 2 n H \;\; {\rm and} \;\; \int_X c_2(T_X(n))    = n^2 H^2 + \int_X c_2(T_X)    = 2 n^2 (g-1) + 24,$$as it follows from \eqref{eq:dualchern} and from the facts that 
$c_0(T_X) = 1$, $c_1(T_X) =0$, $\int_X c_2(T_X) = e(X) = 24$, $H^2 = 2 (g-1)$. Moreover, $h^2(X, T_X(n)) = 0$; indeed $T_X \cong \Omega_X^1$ so, by Serre duality, 
$h^2(T_X(n)) = h^0(T_X(-n))$, the latter being zero because $T_X(-n)\subseteq T_X$ and $T_X$ is not effective by \eqref{eq:tangproperties}--(i). 

Hirzebruch--Riemann--Roch formula (cf. ~\cite[Corollary~15.2.1]{Fu}) therefore reads 
$$\chi(T_X(n)) = h^0(T_X(n)) - h^1(T_X(n)) = \int_X \frac{c_1(T_X(n))^2}{2} - \int_X c_2(T_X(n))    + 2 \,{\rk}(T_X(n))  $$
$$= \frac{(2nH)^2}{2} - (2n^2(g-1) + 24) + 4 = \frac{4n^2H^2}{2} - 2n^2(g-1) - 20 = 4n^2(g-1) - 2 n^2 (g-1) - 20.$$Thus 
\begin{equation}\label{eq:sectionsTn}
h^0(T_X(n)) = 2 n^2 (g-1) - 20 + h^1(T_X(n)) \geq 2 n^2 (g-1) - 20.
\end{equation}
\end{remark}

Using the previous computations, we prove the following useful Lemma.

\begin{lemma}\label{lem:effectivness}  Let $(X, H)$ be a very--general $K3$ surface of polarization $g \geqslant 2$. Then, for any $g \geqslant 2$, there exists an integer 
$n_0(g)$, depending on $g$, for which the vector bundle $T_X(n)$ is effective, for any $n \geqslant n_0(g)$, where the values of~$n_0(g)$ according to the genus $g$ are the following:
\begin{equation}\label{eq:table}
\begin{array}{|c|c|c|c|c|c|c|c|c|c|c|c|}
\hline
g & 2 & 3 & 4 & 5 & 6 & 7 & 8 & 9  & 10 & 11 & \geq 12\\\hline
%\left\lceil \sqrt{\frac{10}{g-1}} \;\;\right\rceil & 3,16 & 2,23 & 1,82 & 1,58 & 1,41 & 1,29 & %1,19 & 1,11\\\hline
n_0(g) & 4 & 3 & 2 & 2 & 2 & 2 & 2 & 2  & 1 & 2 & 1\\\hline 
\end{array}
\end{equation}
\end{lemma}

\begin{proof} Notice that, for $g=10$ and $g \geqslant 12$, Proposition \ref{prop:Beauville} gives always $h^0(T_X(1)) \geqslant 1$, i.e. $n_0(g)=1$ in all these cases. 
Indeed, for $g=10$, Proposition \ref{prop:Beauville}--(b) directly gives $h^0(T_X(1)) =1$;  for $g=12$, \eqref{eq:sectionsTn} applied for $n=1$ gives 
$h^0(T_X(1)) = 2 + h^1(T_X(1)) \geq 3$, the latter inequality following from 
Proposition \ref{prop:Beauville}--(d); at last, for $g\geq 13$, Proposition~\ref{prop:Beauville}--(c) gives $h^1(T_X(1)) =1$ so, by \eqref{eq:sectionsTn}, 
we have $h^0(T_X(1)) = 2 (g-1) - 20 \geq 2 (12) - 20 = 4$.

For $g=11$, differently than above, Proposition \ref{prop:Beauville}--(a) and (c) give $h^0(T_X(1)) = 
h^1(T_X(1)) =0$. On the other hand, formula \eqref{eq:sectionsTn} for $n=2$ gives 
$h^0(T_X(2)) = 60 + h^1(T_X(2))  \geq 60$ so $n_0(11) =2$. 

We are left  with the low--genus cases, i.e. $2 \leq g \leq 9$. From \eqref{eq:sectionsTn}, for any integer $k \geqslant 0$ we get 
$h^0(T_X(k)) \geq 2 k^2 (g-1) -20$, the left--side member being positive as soon as $k^2 > \frac{20}{2(g-1)} = \frac{10}{g-1}$. 
We therefore set $n_0(g) \coloneqq \left\lceil \sqrt{\frac{10}{g-1}} \;\;\right\rceil$, which gives values as in \eqref{eq:table}. 
\end{proof}

\begin{theorem}\label{thm:tangK3} Let $(X, H)$ be a very--general $K3$ surface of polarization $g \geqslant 2$. Then $T_X(n)$ is a big vector bundle if:
\begin{enumerate}
\item $n \geq 5$, for $g =2$ ;
\item $n \geq 4$, for $g=3$ ;
\item $n \geq 3$, for $4 \leq g \leq 9$ or $g=11$ ;
\item $n \geq 2$, for $ g \geq 10$ but $g \neq 11$.
\end{enumerate} For any pair $(g,n)$ as above, let $v_{(g,n)}$ be the Mukai vector $v(T_X(n))$.  Then $M_H(v_{(g,n)})^s \neq \emptyset$. Moreover, 
$M_H(v_{(g,n)})^s$ is a smooth, quasi-projective variety of dimension $90$ 
whose general element parametrizes  rank $2$ big vector bundles on $X$ of Mukai vector $v_{(g,n)}$. 
\end{theorem}
\begin{proof}  We first focus on bigness. From Lemma \ref{lem:effectivness}, we know that $T_X(n_0(g))$ is effective, where $n_0(g)$ is a positive integer as in \eqref{eq:table}.  
We can therefore apply Proposition \ref{prop:Laz}, with $E \coloneqq T_X(n_0(g))$, $m=1$ and $A = k H$, $k \geq 1$ any integer, to get that $T_X(n)$ is big for any integer 
$n:= k+1 \geq n_0(g) +1$.

Concerning stability, from \eqref{eq:tangproperties}-(iii), \eqref{eq:modspacetang} and the fact that $\mu_H$-stability is preserved under tensor product with line bundles, we deduce that, for all pairs $(g,n)$ as in the statement, $M_H(v_{g,n})^s$ is not empty and it
is a smooth quasi-projective variety of dimension~$90$. The assertion on $\dim M_H(v_{(g,n)})^s = 90$ follows by a direct computation: indeed from Proposition \ref{prop:modspace} one has $\dim M_H(v_{(g,n)})^s = 2 + \langle  v_{g,n}, v_{g,n} \rangle$, where 
$$v_{g,n} = v(T_X(n)) = \left(2 , c_1(T_X(n)), \int_X \left(\frac{c_1(T_X(n))^2}{2} - c_2(T_X(n)) \right)    + 2\right) = \left(2, \; 2nH, \; 2 n^2(g-1) -22\right),$$the last equality following from Remark~\ref{gencomp} above. It is straightforward to compute that $\langle  v_{g,n}, v_{g,n} \rangle =88$ so, by Proposition \ref{prop:modspace}, we get $\dim(M_H(v_{(g,n)})^s) = 90$.\footnote{The latter equality more intrisically follows from the fact that the operation of tensor product $- \otimes H^{\otimes n}$ establishes, for any integer $n \geqslant 1$, an isomorphism between the moduli space $M_H^s ((2, 0, - 22))$  and  the moduli space $M_H(v_{g,n})^s$, the isomorphism sending $[T_X]$ to $[T_X(n)]$. Thus, from~\eqref{eq:modspacetang} $M_H(v_{(g,n)})^s$  has the same dimension as $M_H((2, 0, - 22))^s$ which is of dimension $90$.}

At last, since bigness is an open condition and since $T_X(n)$ is big for any $n \geqslant n_0(g)+1$, the general stable bundle 
parametrized by $M_H(v_{(g,n)})^s$ is therefore big for any pair $(g,n)$ as above.
\end{proof}

%%%%%%%%%%%%%%%%%%%%%%%%%%%%%%%%%%%%%%%%%%
%%%%%%%%%%%%%%%%%%%%%%%%%%%%%%%%%%%%%%%%%%
%%%%%%%%%%%%%%%%%%%%%%%%%%%%%%%%%%%%%%%%%%

\subsection{Big and stable Mukai-Lazarsfeld vector bundles on $K3$ surfaces}\label{oddrank} Smooth curves on polarized $K3$ surfaces,  in particular their
Brill-Noether theory, play a fundamental role in Algebraic Geometry. Indeed the Brill-Noether theory of these curves is deeply 
connected to the geometry of the surface, to modular properties of curves on $K3$ surfaces, as well as it
is fundamental to prove results on smooth curves with  general moduli with no use of degeneration techniques 
(cf.~\cite{Laz0}). 

Lazarsfeld's approach to Brill--Noether theory without degenerations in \cite{Laz0} uses vector-bundle techniques on $X$; given  
$(X,H)$ a primitively polarized $K3$ surface of genus $g \geq 3$, a general curve $C \in |H|$ and a complete linear series $|A| =
g^{r-1}_d$ on $C$, with suitable properties of global generations on $A$ and of integrality assumptions on any member of $|H|$, 
Lazarsfeld associates a rank $r$ vector bundle $E := E_{C,A}$ on $X$ to the triple $(X,C,A)$, the vector bundle $E$ depending on the choice of 
$C \in |H|$ and of the line bundle $A$ on $C$. This vector bundle $E$ encodes several properties of Brill-Noether and Petri's theory 
of the scheme $W^{r-1}_d(C)$, parametrizing {\em special linear series} on $C$. Here we will briefly recall Lazarsfeld's approach in \cite{Laz0} as it will allow us to also construct families of stable and big vector bundles of any rank $r \geqslant 2$ on a very--general polarized $K3$ surface.

Let $(X,H)$ be a smooth, polarized, projective $K3$ surface of genus $g \geqslant 3$. Given a curve $C$ and positive 
integers $d$ and $r$, consider the {\em Brill--Noether locus} 
$$W^{r-1}_d (C):= \{ A \in \; {\rm Pic}^d(C) \; | \; h^0(C, A) \geqslant r \} \subseteq {\rm Pic}^d(C)$$and its 
subscheme $$V^{r-1}_d(C) \subseteq  W^{r-1}_d(C)$$defined to be the non-empty, open  subset of $W^{r-1}_d(C)$ consisting of line bundles $A$ on $C$ such that:
\begin{itemize}
\item[(i)] $h^0(C,A) = r$, $\deg(A) = d$, and 
\item[(ii)] both $A$ and $\omega_C \otimes A^{\vee}$ are globally generated on $C$ (where $\omega_C$ denotes the canonical bundle of $C$).
\end{itemize}In this set--up, for any smooth curve $C \in |H|$ and any line bundle $A \in V^{r-1}_d(C)$ 
one associates to the pair $(C, A)$ a rank $r$ vector bundle $E := E_{C,A}$ on $X$ as follows:
since $A$ is globally generated, we have a canonical surjective map
$$ev_{C,A} : H^0(C,A) \otimes {\mathcal O}_X \twoheadrightarrow A$$of ${\mathcal O}_X$-modules (thinking $A$ as a sheaf on $X$);
thus, $\ker(ev_{C,A})$ is a rank~$r$ vector bundle on $X$, therefore, one sets$$E = E_{C,A} := \ker(ev_{C,A})^{\vee}$$(for details,
cf. \cite[\S1]{Laz0}). This gives rise to the exact sequence on $X$:
\begin{equation}\label{eq:L11}
0 \to E^{\vee} \to H^0(C,A) \otimes {\mathcal O}_X \to A \to 0.
\end{equation}Dualizing \eqref{eq:L11}, one gets
\begin{equation}\label{eq:L12}
0 \to H^0(C,A)^{\vee} \otimes {\mathcal O}_X \to E \to \omega_C \otimes A^{\vee} \to 0,
\end{equation}since ${\mathcal Ext}^1_{{\mathcal O}_X} (A, {\mathcal O}_X) \cong \omega_C \otimes A^{\vee}$ (cf. \cite[Lemma\;7.4, p.\;242]{H}). The vector bundle $E$ is called the 
{\em Mukai-Lazarsfeld vector bundle} associated to the pair $(C, A)$. If, as it is customary, one considers the {\em Brill-Noether number}:
\begin{equation}\label{eq:rhoA}
g - h^0(C,A) h^1(C,A) = g - r ( r -1 + g - d) = : \rho(g,r-1,d) 
\end{equation}from \eqref{eq:L11}, \eqref{eq:L12} and the fact that
$X$ is regular with $\omega_X \cong {\mathcal O}_X$, one has (cf.\;\cite[\S\,1]{Laz0}):
\begin{enumerate}
\item[(E1)] $E$ is a rank~$r$, globally generated vector bundle on $X$,
\item[(E2)] $c_1(E) = H, \; \int_X c_2(E)    = \deg(A) = d,$
\item[(E3)] $h^0(X, E^{\vee}) = h^2(X, E) = 0, \;  h^1(X, E^{\vee}) = h^1(X, E) = 0$,
\item[(E4)] $h^0(X, E) = h^0(C,A) + h^1(C,A) = 2r + g - d - 1$;
\item[(E5)] $\chi(X, E \otimes E^{\vee}) = 2 - 2 \rho(g,r-1,d) $ (cf. \cite[$(iv)$, p. 302]{Laz0}). If moreover $E$ is {\em simple}, more precisely one has 
$h^0(X, E \otimes E^{\vee}) = h^2(X, E \otimes E^{\vee}) = 1$ and $h^1(X, E \otimes E^{\vee}) = 2 \rho(g,r-1,d)$.
\end{enumerate}

Another fundamental property of the vector bundle $E$ is given by the following:

\begin{lemma}\label{lem:L13} (cf. \cite[Lemma 1.3]{Laz0}) If $E$ has non-trivial endomorphisms,
i.e. if $h^0(X, E \otimes E^{\vee}) \geq 2$, the linear system~$|H|$ contains a reducible (or multiple) curve.
\end{lemma}

Thus, in particular, we have the following:

\begin{proposition}\label{prop:Lstab} 
Let $(X,H)$ be a primitively polarized $K3$ surface, such that $H^2 >2$ and that $|H|$ contains neither reducible nor non-reduced curves.  
Let $C \in |H|$ be any smooth curve and let $A \in V^{r-1}_d(C)$. Then:
\begin{itemize}
\item[(i)] $E$ is a simple bundle on $X$.
\item[(ii)] If moreover $(X,H)$ is assumed to be very--general, then $E$ is a $\mu_H$-stable bundle on $X$. Thus, setting 
$v = v(E) := (r, H, g-1 - d + r) $ as the Mukai vector of $E$, then  the moduli space $M^s_H(v)$ is smooth, irreducible and 
such that $\dim ( M^s_H(v)) = 2 \rho(g,r-1,d)$.
\end{itemize}
\end{proposition}

\begin{proof}\text{} 
	
\noindent (i) The fact that $E$ is simple directly follows from the assumption on $|H|$, containing neither reducible nor non-reduced curves, 
and from Lemma \ref{lem:L13}.

\noindent
(ii) If $(X,H)$ is assumed to be very--general, then in particular ${\rm Pic}(X) \cong \mathbb{Z}[H]$. Stability of 
$E$ is then proved e.g. in \cite[Prop.\;4.5]{FF}.  Moreover since $ v = v(E) = (r, H, g-1 - d + r)$ is primitive and  since 
$H$ is the generator of ${\rm Pic}(X)$, then Proposition\;\ref{prop:modspace}--(ii) implies irreducibility of $M^s_H(v)$, whose dimension is given 
by $2 + \langle v , v \rangle = 2 \rho(g,r-1,d).$
\end{proof}

In this general set--up, we have the following result.

\begin{theorem}\label{prop:oddvect} Let $(X,H)$ be a very--general polarized $K3$ surface of genus $g \geqslant 3$. For any triple $(g, r-1, d)$ of positive integers such that 
$ d < 2g-2$ and $\rho(g,r-1,d) \geqslant 0$, there exists a vector bundle $E$ on $X$ of rank $r$, with $c_1(E) = H$ and $c_2(E) = d$ which is globally generated and $\mu_H$-stable on $X$. Setting 
$v = v(E) = (r, H, g-1-d +r)$ as the Mukai vector of $E$, then $M_H(v)^s$ is smooth, irreducible of dimension $\dim (M_H(v)^s) = 2\rho(g,r-1,d)$.
When in particular $\rho(g,r-1,d) = 0$, then $M_H(v)^s$ consists of the single, reduced point $\{E\}$, which is also a big bundle on $X$. 
When otherwise $\rho(g,r-1,d) > 0$, the general bundle parametrized by $M_H(v)^s$ is globally generated and big on $X$. 
\end{theorem}

\begin{proof} The existence of $E$ of rank $r$, with $c_1(E) = H$ and $c_2(E) = d$, the fact that $E$ is globally generated and moreover 
$\mu_H$-stable are direct consequences of $(E1)-(E2)$ and of Proposition \ref{prop:Lstab}--(ii) above. 

Moreover, by \eqref{eq:segre}, one has that the Segre class $ \int_X s_2(E) = \int_X (c_1^2(E)-c_2(E))    = H^2-d=2g-2-d$ therefore, since $d < 2g-2$ by assumption, $ \int_X s_2(E) = (-1)^2 \int_X s_2(E)    >0$. Since $E$ is globally generated, by Proposition \ref{prop:bigness}, the bundle $E$ is therefore also big. 

The case $\rho(g,r-1,d) = 0$ clearly gives an exceptional Mukai vector $v = v(E)$, i.e. $\langle v, v \rangle = -2$, therefore $M_H(v)^s$ consists only 
of the single reduced point $[E]$, as it follows from Proposition \ref{prop:modspace}--(i), which we have already remarked to be globally generated and big. 

When otherwise $\rho(g,r-1,d) > 0$, then  $M_H(v)^s$ is smooth, irreducible of positive dimension $2 \rho(g,r-1,d)$, as it follows  
from the facts that ${\rm Pic}(X)$ is cyclic generated by $H$, $v$ is primitive and from Propositions \ref{prop:modspace}--(i) and \ref{prop:Lstab} above. At last, since global generation and 
bigness are both open conditions in our situation, the general bundle parametrized by $M_H(v)^s$ is therefore globally generated and big on $X$. 
\end{proof}

\begin{remark}\label{rem:particularcase} As a very particular case of the previous description, one can consider triples $(g,r,d)$ for which 
$\rho(g,r-1,d) = 0$, equivalently $\langle v , v \rangle = -2$. In all such cases, by Brill-Noether theory on $C$ and Lazarsfeld's results in \cite{Laz0}, 
on a general member $C \in |H|$ the Brill--Noether locus $W^{r-1}_d(C) = V^{r-1}_d(C)$ consists of finitely many reduced points whereas, from  
Proposition\;\ref{prop:modspace}--(i),  for any associated Mukai vector $v = v(E_{C,A})$, the moduli space 
$M^s_H(v)$ consists of a single reduced point, i.e. for any general member $C \in |H|$ and any line bundle $A \in V^{r-1}_d(C)$ on $C$, 
all Mukai--Lazarsfeld vector bundles $E = E_{C,A}$ as above are isomorphic each other. Cases for which $\rho(g, r-1, d) = 0$ belong to a wider class of $\mu_H$--stable vector bundles $F$ studied in \cite[Thm.\,2.1]{Kul} and 
\cite[Prop.\,2.1]{MR}, for which $$\gcd (\rk(F), \int_X c_1(F) \cdot H    ) =1 \;\;\; {\rm and}\;\;\; 2 \rk(F) \int_X c_2(F) - (\rk(F)-1) \int_X c_1(F)^2    = 2 (\rk(F)^2-1).$$Indeed, if we set 
$c_1 (F)\coloneqq H$ and $\int_X c_2(F)    \coloneqq d \geqslant 0$, we have 
$$v (F)= (\rk(F), H, g-1 + \rk(F) -d)$$as $ \int_X c_1(F) \cdot H    = H^2 = 2 (g-1)$ and the condition $\gcd (\rk(F),  \int_X c_1(F) \cdot H    ) =1$ implies in particular 
that $\rk(F)$ is certainly odd (moreover not divisible by all prime divisors of $(g-1)$), in particular $\rk(F) \geqslant 3$. Furthermore, 
condition $2 \rk(F) \int_X c_2(F)    - (\rk(F)-1) \int_X c_1(F)^2     = 2 (\rk(F)^2-1)$ above, reads in this case as $\rk(F)d =(\rk(F)-1)(g+\rk(F))$. 

Particular triples satisfying these numerical conditions are e.g. $$(g,r,d) \in \{(6,3,6), (9,3,8), (10,5,12), (12,3,10),\ldots\}.$$Under these numerical conditions, $v(F)^2 \coloneqq \langle v, v \rangle = -2$ and such a Mukai vector is called {\em exceptional} (cf.\,e.g.\,\cite[\S\,2]{MR}). From \cite[Thm.\,2.1]{Kul} and \cite[Prop.\,2.1]{MR}, for any triples $(\rk(F), d, g) = (r, d, g)$ satisfying 
$c_1(F) = H$ and the previous numerical assumptions, there exists a $\mu_H$-stable (odd) rank~$r$ 
vector bundle $F$ on $X$ with Mukai vector $v=(r, H, g-1 + r -d)$, with $rd=(r-1)(g+r)$, and $M_v(X)^s = \{ F \}$ is a single reduced point. Any vector bundle isomorphic to such a 
$F$ is called {\em exceptional}, because it is related to an exceptional Mukai vector $v = v(F)$ (cf.\,\cite{Kul,MR}). 
The previous construction shows that bundles $F$ as above arise as Mukai--Lazarsfeld vector bundles $F = E_{C,A}$ associated to pairs $(C, A)$ where 
$C \in |H|$ is a general member whereas $A$ is a line bundle on $C$ of degree $d$, with $h^0(C, A) = r$ and such that $\rho(g, r-1, d) = 0$.  
\end{remark}

%%%%%%%%%%%%%%%%%%%%%%%%%%%%%%%%%%%%%%%%%%
%%%%%%%%%%%%%%%%%%%%%%%%%%%%%%%%%%%%%%%%%%
%%%%%%%%%%%%%%%%%%%%%%%%%%%%%%%%%%%%%%%%%%

\subsection{Big and stable Ulrich bundles on $K3$ surfaces}\label{evenrank}  Let $(X,H)$ be any primitively polarized K3 surface of genus $g\geq 2$, where $H$ is globally generated and ample. As in the previous sections, for any vector bundle $E$ on $X$ and any integer $n$, we will simply denote by $E(n)$ the tensor product $E \otimes H^{\otimes n}$. 

A rank~$r$ vector bundle $E$ on $X$ is said to be an \emph{Ulrich bundle} if it satisfies the conditions:
\[
	H^\ast(X, E(-1))= 0 \text{ and } H^\ast(X,E(-2)) = 0.
\]In particular, $H^i(X, E(-i)) = 0$ for all integers $i>0$ so,  by \cite[Def.\;1.8.4,\;p. 100]{Laz1}, $E$ is {\em (Castelnuovo-Mumford) $0$--regular with respect to} $H$ and, 
by \cite[Thm.\;1.8.5-(i),\;p. 100]{Laz1}, $E$ is therefore globally generated on $X$ and $H^i(X, E) = 0$ for all $i>0$. Using the Riemann--Roch--Hirzebruch formula~\cite[Corollary~15.2.1]{Fu} and equation~\eqref{eq:segre} we compute the Euler--Poincar\'e polynomial of $E$:
\[
	\chi(X, E(n)) = r (g - 1) n^2  + n \int_X c_1(E) \cdot H 
					+ \frac{1}{2} \int_X (c_1(E)^2  - 2 c_2(E)) + 2r.
\]
By assumption $\chi(X, E(-1)) = \chi(X, E(-2))= 0$, so $\chi(X, E(n)) = r (g - 1) (n + 1) (n + 2)$.
We get by identification:
\[
	\int_X c_1(E) \cdot H = 3 r (g - 1),
	\quad
	\int_X (c_1(E)^2  - 2 c_2(E)) = 4r(g - 2),
	\quad
	h^0(X, E) = 2 r (g - 1).
\]
If we assume $(X, H)$ to be very--general, in particular $\Pic(X) \cong \bZ[H]$, then $c_1(E) = \lambda H$, for some $\lambda \in {\mathbb Z}$, so we get  
$2\lambda = 3r$ thus necessarily $r$ is even and $c_1(E) = \frac{3 r}{2} H$ . Assuming this, it follows that $ \int_X c_1(E)^2    = \frac{9 r^2}{2} (g - 1)$, $ \int_X c_2(E)    = \frac{9 r^2}{4} (g - 1) - 2 r (g - 2)$, and finally using~\eqref{eq:IID4} we get:
\begin{equation}\label{eq:segulrich}
\int_X s_2(E)    = \frac{9 r^2}{4} (g - 1) + 2 r (g - 2) >0 \quad \forall g\geq 2.
\end{equation} From global generation of $E$ and from Proposition~\ref{prop:bigness}, the positivity of $ \int_X s_2(E)    = (-1)^2 \int_X s_2(E)   $ in \eqref{eq:segulrich} implies that Ulrich bundles $E$ on $K3$ surfaces are therefore big. Taking into account \cite[Thm.\;0.5]{AFO}, we get:  

\begin{theorem}\label{thm:ultrich} For any integer $h \geqslant 2$, let $(X,H)$ be a very--general polarized $K3$ surface, with polarization $H$ of genus 
$g = h+1$. Then, for every integer $a \geqslant 1$, there exists an $(8a^2 + 2a^2h +2)$--dimensional family of $\mu_H$--stable Ulrich bundles $E$ on $X$ such that 
$$\rk(E) = 2a, \;\; c_1(E) = 3 a H,\;\; \int_X c_2(E)    = 9a^2h - 4a(h-1).$$
Setting $v := v(E) = (2a, 3aH, 2a(2h-1))$ the Mukai vector of any such bundle $E$, 
the moduli space $M^s_H (v)$ is irreducible and its general point corresponds to a $\mu_H$--stable Ulrich bundle of rank 
$r=2a$ which is also globally generated and big.
\end{theorem}
\begin{proof} For the proof of the existence of Ulrich bundles as in the statement, we refer the reader for full details to the original paper \cite{AFO}. 
Here we briefly recall basic steps for the construction of the aforementioned bundles. 

For the case $a=1$, i.e. for $\rk(E) =2$, Ulrich vector bundles $E$ have been constructed once again via Mukai--Lazarsfeld bundles 
as in \S\ref{oddrank}. Namely any rank~$2$ Ulrich bundle $E$ as in the statement is given by a Mukai--Lazarsfeld vector bundle 
$E := E_{C,A}$ where the curve $C \in |{\mathcal O}_X(3)|$ is a general cubic section of $X$ (so not anymore a general member in the linear system $|H|$ of the generator of ${\rm Pic}(X)$ 
as in \S\ref{oddrank}),  which is a curve of genus $9h + 1$ and of degree $6h$ in $\mathbb{P}^{h+1}$, whereas $A $ is a line bundle 
on $C$ giving rise to a complete and base--point--free pencil of degree $5h+4$ on $C$, i.e. $A\in W^1_{5h+4}(C)$. Since the curve $C$ is in this case far from being {\em Brill-Noether general}, 
showing that a general cubic section $C$ of $X$ actually carries a pencil $|A| = g^1_{5h+4}$ inducing an Ulrich bundle $E$ on $X$ has been translated by the authors of \cite{AFO} in a variational problem which has also been used in their proof of Green's conjecture for curves on arbitrary $K3$ surfaces.

For cases $a >1$, taking direct sums of Ulrich bundles as in the rank~$2$ case above, the authors then construct splitting Ulrich bundles of any even rank $r=2a$ 
using direct sums of rank~$2$ bundles and then they show that, for a very--general $K3$ surface, these direct sums can be deformed to stable unsplitting 
Ulrich bundles on $X$ of the same rank $r=2a$.  

With the above set--up, for any $a \geqslant 1$, one has $$\rk(E) = 2a, \;\; c_1(E) = 3 a H,\;\; \int_X c_2(E)    = 9a^2h - 4a(h-1)$$and it is a straightforward computation showing that 
$$ \int_X \left(\frac{c_1(E)^2}{2} - c_2(E) \right)   + \rk(E) = 2 a (2h-1)$$ so the Mukai vector of any such $E$ is $v = (2a, 3aH, 2a(2h-1))$ as stated. The irreducibility follows from the fact that $(X, H)$ is very general and from Proposition~\ref{prop:modspace}--(iii). Finally, since the general point of  $M_H^s(v)$ is proved to be an Ulrich bundle, it is also globally generated and big, as it is $0$-regular 
with $(-1)^2 \int_X s_2(E)    = \int_X s_2(E)    >0$, as shown in \eqref{eq:segulrich}, and then by applying Proposition~\ref{prop:bigness}. 
\end{proof}

%%%%%%%%%%%%%%%%%%%%%%%%%%%%%%%%%%%%%%%%%%
%%%%%%%%%%%%%%%%%%%%%%%%%%%%%%%%%%%%%%%%%%
%%%%%%%%%%%%%%%%%%%%%%%%%%%%%%%%%%%%%%%%%%

\section{Generalities on tautological vector bundles on $Hilb^k$ of a $K3$ surface}\label{genbundlesHK} In this section we will introduce some generalities on {\em tautological vector bundles} 
on varieties $X^{[k]} \coloneqq Hilb^k(X)$, i.e. the {\em Hilbert scheme} parametrizing $0$--dimensional subschemes of length $k$ on a polarized surface $(X, H)$, for any integer $k \geqslant 2$ (cf.\,e.g. \cite{Bos,Dan1,Dan2,Dan3,Kr,Sca1,Sca2,Sta}). These preliminaries will be used in \S\ref{bundlesHK}, where we focus on the case of $(X, H)$ a very--general $K3$, so that $X^{[k]}$ turns out to be a {\em Hyper--K\"ahler variety}, and where we consider tautological bundles arising from those in \S\ref{bundlesK3}.

Let $X$ be any smooth, projective complex surface. Since $X^{[k]}$ is a fine moduli space, it is endowed with a universal family 
$\Xi_k \subset X \times X^{[k]}$, together with the two natural projections 
$$ X \overset{\pi_X}{\longleftarrow} \Xi_k \overset{\pi_{X^{[k]}}}{\longrightarrow} X^{[k]},$$the map $\pi_{X^{[k]}}$ being 
flat of finite degree $k$. One can therefore associate to every coherent sheaf $F$ on X the so called 
{\em tautological sheaf associated to $F$} on $X^{[k]}$, which is defined to be 
\begin{equation}\label{eq:Fk}
F^{[k]} \coloneqq {\pi_{X^{[k]}}}_* (\pi_X^*(F)).
\end{equation} If $F$ is locally free of rank $m$, then the tautological bundle $F^{[k]}$ is locally free of rank
$km$ with fibres 
\begin{equation}\label{eq:Fkbis}
(F^{[k]})(\xi) \cong H^0(F|_{\xi}), \;\; \forall \;\; \xi \in X^{[k]}
\end{equation}(cf. \cite[Rem.\;3.6]{Kr}). Denoting by $X^k$ the $k$-th cartesian product $X^k = X \times \cdots \times X$ and by $X^{(k)} = X^k/Sym(k)$ the $k$-th symmetric product 
of $X$, it is well--known that $X^{[k]}$ is a resolution of the singularities of $X^{(k)}$ via the {\em Hilbert--Chow morphism}
$$\mu : X^{[k]} \to X^{(k)}, \;\; \xi \stackrel{\mu}{\longrightarrow} \sum_{x \in {\rm Supp}(\xi)} {\rm length}_{\xi} (x) \, x.$$

Let $\pi_i : X^k \to X$ denote the $i$-th natural projection. For any 
$L \in {\rm Pic}(X)$ the line bundle $L^{\boxtimes k} \coloneqq \otimes_{i=1}^k \pi_i^*(L)$ descends to a line bundle $\mathcal{L}\in {\rm Pic} (X^{(k)})$.  Thus, one can define the natural morphism: 
\begin{equation}\label{eq:Dk}
D_k : {\rm Pic} (X) \to {\rm Pic} (X^{[k]}), \;\;\; L \stackrel{D_k}{\longrightarrow} D_k(L) \coloneqq \mu^*(\mathcal{L})
\end{equation} (cf.\;e.g.\;\cite{Kr}) which is injective and which gives, under the assumption that $H^1(X,\mathcal{O}_X)  = 0$ (see~\cite[Theorem~6.2\&Corollary~6.3]{Fogarty}):
\begin{equation}\label{eq:Pics}
 {\rm Pic} (X^{[k]}) = D_k({\rm Pic} (X)) \oplus \mathbb{Z} [\Delta],
\end{equation} 
where $\Delta \coloneqq \det ({\mathcal O}_X^{[k]}) $ and $c_1(\Delta) = - \frac{1}{2} {\mathcal E}$, where ${\mathcal E}$ denotes the {\em $\mu$-exceptional divisor}     (cf. \cite[\S\,5, p.11]{EGL}). In this set--up, one has:
\begin{equation}\label{eq:Krugiso}
H^0(X^{[k]}, F^{[k]} \otimes D_k(L)) \cong H^0(X, F \otimes L) \otimes S^{k-1} \left( H^0(X,  L) \right)
\end{equation} (cf.\;\cite{Dan2} or \cite[formula (1),\;p.\,2]{Kr}). When $L \in {\rm Pic} (X)$ is an effective and ample line bundle on $X$, it follows that $D_k(L)$ is a big and nef line bundle on $X^{[k]}$ (cf. \cite[p.\,5]{Sta}). Moreover, if $D_k(L)$ is 
effective, for any effective divisor $D \in |D_k(L)|$, then $D$ is set--theoretically described as $D= \left\{\xi \in X^{[k]} \; | \;  \xi \cap {\rm Supp}(D) \neq \emptyset \right\}$ (cf. \cite[p.\,5]{Sta}).

%%%%%%%%%%%%%%%%%%%%%%%%%%%%%%%%%%%%%%%%%%%%%%%%%%%%%%%%%
%%%%%%%%%%%%%%%%%%%%%%%%%%%%%%%%%%%%%%%%%%%%%%%%%%%%%%%%%
%%%%%%%%%%%%%%%%%%%%%%%%%%%%%%%%%%%%%%%%%%%%%%%%%%%%%%%%%

\subsection{Tautological bundles and stability}\label{S:taustability} The notion of slope--(semi)stability can be formally extended to big and nef line bundles (cf.\;\cite[p.\;437]{Tsu}). Indeed, using notation and terminology as above, one has the following: 

\begin{theorem}\label{thm:Stapleton} (cf.\;\cite[Thm.\,1.4]{Sta}) Let $(X,H)$ be a smooth, irreducible polarized surface, where $H$ a globally generated and ample line bundle on $X$.  Let 
$k \geqslant 2$ be any integer.  Let $F$ be a rank~$r$ vector bundle on $X$, where $r \geqslant 1$. 
If $F \neq \mathcal O_S$ and if moreover $F$ is $\mu_H$-stable on $X$, then the tautological bundle $F^{[k]}$ is $\mu_{D_k(H)}$-stable on $X^{[k]}$. 
\end{theorem}

%%%%%%%%%%%%%%%%%%%%%%%%%%%%%%%%%%%%%%%%%%
%%%%%%%%%%%%%%%%%%%%%%%%%%%%%%%%%%%%%%%%%%
%%%%%%%%%%%%%%%%%%%%%%%%%%%%%%%%%%%%%%%%%%

\subsection{Segre classes of tautological bundles}\label{segre_tautHilb}

Consider the incidence variety $X^{[k-1,k]}\subset X^{[k-1]}\times X\times X^{[k]}$ parametrizing triples $(\xi,x,\xi')\in X^{[k-1]}\times X\times X^{[k]}$ such that $\xi\subset \xi'$ with residual subscheme supported at the point $x$. This variety has dimension $2k$. We denote the projections as follows:
\[
\xymatrix{ & X^{[k-1,k]}\ar[r]^\rho\ar[dl]^\varphi\ar[dr]^\psi & X\\ X^{[k-1]} && X^{[k]}}
\]
Given a triple $(\xi,x,\xi')\in X^{[k-1,k]}$, either $x$ is already in the support of $\xi$, meaning that $\xi'$ is obtained by thickening $\xi$ at the point~$x$, or not. We denote by $ \mathfrak{E}   $ the set of those triples such that $x$ is in the support of $\xi$. 

If $\cI_{\Xi_k}$ denotes the ideal sheaf of the universal family $\Xi_k \subset X \times X^{[k]}$, we will denote by $\bP(\cI_{\Xi_k})$ the {\em blowing--up of $X \times X^{[k]}$ w.r.t. the coherent sheaf of ideals $\cI_{\Xi_k}$} in the sense of \cite[Def.,\;Cap.\,7,\;p.163]{H}. Then, there are isomorphisms
$$X^{[k-1,k]}\cong \bP(\cI_{\Xi_k})\cong \mathrm{Bl}_{\Xi_k}(X\times X^{[k]})$$
(cf.\,\cite[\S\;1.2]{Lehn}). From these isomorphisms, we see that the $\mathfrak{E}   $ turns out to be the exceptional divisor of the blow--up on the right hand side.

Let $\bH_k\coloneqq \bigoplus_{j=0}^{4k} H^j(X^{[k]},\bQ)$ and $\bH\coloneqq\bigoplus_{k\geq 0} \bH_k$. For any $\alpha\in H^\ast(X,\bQ)$ and any $i\geq 1$, the Nakajima operator $q_i(\alpha)\in\End(\bH)$ is defined by
\[
q_i(\alpha)(\beta)\coloneqq \psi_\ast(\varphi^\ast \beta\cdot \rho^\ast \alpha), \quad \forall \beta \in \bH.
\]
Following always~\cite{Lehn}, we denote by $\partial\in\End(\bH)$ the operator acting on $\bH_k$ by cup-product with the class $c_1(\cO_X^{[k]})$ and
 we define the \emph{derived Nakajima operators} by:
\[
q_i'(\alpha)\coloneqq [\partial, q_i(\alpha)] = \partial \circ q_i(\alpha) - q_i(\alpha)\circ\partial.
\]
(Since the K3 surface $X$ has no odd cohomology, there is no alternate sign in the definition of the Lie bracket).

Starting from $s(F^{[0]}) = 1$, the total Segre classes can be computed recursively using the following formula which, in particular, extends either to 
non--primitive line bundles or to higher--rank vector bundles computations on Segre classes in~\cite{MOP_rank1, MOP, V}:

\begin{proposition}\label{prop:segre_taut}
Let $F$ be a rank $r\geq 1$ vector bundle on a K3 surface $X$. Then for any $k\geq 1$ one has:
\[
s(F^{[k]}) = \frac{1}{k} \left(\sum_{i=0}^{2k} \sum_{j=0}^{\min\{i,2\}} (-1)^{i-j} \binom{r-1+i}{r-1+j} q_1^{(i-j)}(s_j(F))\right)\left(s(F^{[k-1]})\right).
\]
\end{proposition}

\begin{proof}
Following~\cite[equation (11)]{Lehn} (see also~\cite[Lemma~2.1]{EGL}), for any rank $r$ vector bundle $F$ on $X$, we have an exact sequence relating the tautological bundles associated to $F$ on $X^{[k-1]}$ and $X^{[k]}$:
\[
0\longrightarrow \rho^\ast F\otimes\cO_{X^{[k-1,k]}}(- \mathfrak{E}   )\longrightarrow \psi^\ast F^{[k]}\longrightarrow \varphi^\ast F^{[k-1]}\longrightarrow 0.
\]
The basic properties of the total Segre class give:
\begin{equation}\label{eq:rec_segre}
\psi^\ast s(F^{[k]})=\varphi^\ast s(F^{[k-1]})\cdot s\left( \rho^\ast F\otimes\cO_{X^{[k-1,k]}}(-\mathfrak{E}) \right).
\end{equation}
We put $\lambda\coloneqq c_1(\cO_{X^{[k-1,k]}}(-   \mathfrak{E}   ))$. Using formula~\eqref{eq:dualsegre} we get:
\begin{equation}\label{eq:segre_tensor_line}
s\left( \rho^\ast F\otimes\cO_{X^{[k-1,k]}}(-\mathfrak{E}) \right)
= \sum_{i=0}^{2k} \sum_{j=0}^{\min\{i,2\}} (-1)^{i-j} \binom{r-1+i}{r-1+j} \rho^\ast s_j(F)\lambda^{i-j}.
\end{equation}

For any $k$, we denote by $\seg(F^{[k]})\in\End(\bH_k)$ the operator acting on $\bH_k$ by cup-product with the total Segre class~$s(F^{[k]})$. 
For any $\beta\in \bH_{k-1}$, following the same lines as the proof of~\cite[Theorem 4.2]{Lehn} we compute, using \eqref{eq:rec_segre} and~\eqref{eq:segre_tensor_line}:

\begin{align*}
\seg(F^{[k]})\circ q_1(1)\circ \seg(F^{[k-1]})^{-1}(\beta)
&=s(F^{[k]})\cdot\psi_\ast\left(\varphi^\ast\left(\seg(F^{[k-1]})^{-1}(\beta)   \right)\right)\\
&=\psi_\ast \left(\psi^\ast s(F^{[k]}) \cdot \varphi^\ast\left(s(F^{[k-1]})^{-1}\right)\cdot\varphi^\ast(\beta)    \right)\\
&=\psi_\ast \left( \sum_{i=0}^{2k} \sum_{j=0}^{\min\{i,2\}} (-1)^{i-j} \binom{r-1+i}{r-1+j} \rho^\ast s_j(F)\lambda^{i-j}\varphi^\ast(\beta)  \right).
\end{align*}
By~\cite[Lemma~3.9]{Lehn}, for any $\beta\in \bH$ and any $\nu\geq 0$, we have
$q^{(\nu)}_1(\alpha)(\beta) = \psi_\ast(\lambda^\nu\cdot\varphi^\ast \beta\cdot \rho^\ast \alpha)$,
so we obtain:
\[
\seg(F^{[k]})\circ q_1(1)\circ \seg(F^{[k-1]})^{-1}(\beta)
= \sum_{i=0}^{2k} \sum_{j=0}^{\min\{i,2\}} (-1)^{i-j} \binom{r-1+i}{r-1+j} q_1^{(i-j)}(s_j(F)\cdot \beta).
\]
We denote by $\vac$ the unit in the ring $\bH_0$. Recall that $\frac{1}{k!}q_1(1)^k {\bf{1}}$ is the unit in the ring $\bH_k$ for any $k$. 
 By the above formula with $\beta=\frac{1}{(k-1)!}q_1(1)^{k-1} {\bf{1}}$ we get the expected recursion relation:
\[
s(F^{[k]}) = \frac{1}{k} \left(\sum_{i=0}^{2k} \sum_{j=0}^{\min\{i,2\}} (-1)^{i-j} \binom{r-1+i}{r-1+j} q_1^{(i-j)}(s_j(F))\right)(s(F^{[k-1]})).
\]
\end{proof}

In particular, since each operator $q_n^{(\nu)}(\alpha)$ has cohomological degree $2\nu + 2(n-1) + \deg(\alpha)$, we deduce a recursion formula for the $d$-th Segre class $s_d(F^{[k]}) \in \bH_{2d}$ for any $0\leq d\leq 2k$:
\begin{equation}\label{eq:dth_Segre}
	s_d(F^{[k]}) = \frac{1}{k} \left(\sum_{i\geq \max\{0, d-2(k-1)\}}^{d} \sum_{j=0}^{\min\{i,2\}} (-1)^{i-j} \binom{r-1+i}{r-1+j} q_1^{(i-j)}(s_j(F))\right)\left(s_{d-i}(F^{[k-1]})\right).
\end{equation}

\begin{corollary}\label{cor:segre_taut}
Let $F$ be a rank $r\geq 1$ vector bundle on a K3 surface $X$. Then:
\begin{align*}
\int_{X^{[2]}} s_4(F^{[2]})    
&= 
12\binom{r + 3}{4} 
- \frac{1}{2} \binom{r + 2}{2} \cdot \int_{X} s_1(F)^2
- \left(\frac{r^2 +3 r + 3}{2}\right) \cdot \int_{X} s_2(F)
+ \frac{1}{2} \cdot \left(\int_{X} s_2(F)\right) ^ 2;\\
   \int_{X^{[3]}} s_6(F^{[3]})     
&=
-2 \left(4 r ^ 3 + 21 r ^ 2 + 35 r + 20 \right) \binom{r + 2}{3}
+ \frac{1}{3}\binom{r + 2}{2}\left(3 r ^ 2 + 8 r + 6\right) \cdot \int_{X} s_1(F)^2\\
& + \frac{1}{6} \left(6 r ^ 4 + 35 r ^ 3 + 72 r ^ 2 + 61 r + 20\right) \cdot \int_{X} s_2(F)
- \frac{1}{2}\binom{r + 2}{2} \cdot \int_{X} s_1(F)^2\cdot \int_{X} s_2(F)\\
& - \left(\frac{r ^ 2 + 3 r + 3}{2}\right) \cdot \left(\int_{X} s_2(F)\right) ^ 2
+ \frac{1}{6} \cdot\left(\int_{X} s_2(F)\right) ^ 3.
\end{align*}
\end{corollary}

\begin{proof} We apply \eqref{eq:dth_Segre}, starting from $s(F)= 1 + s_1(F) + s_2(F) =  q_1(1)\vac + q_1(s_1(F))\vac + q_1(s_2(F))\vac$:
\begin{align*}
2 s_4(F^{[2]}) &= 
\binom{r+1}{r-1} q_1^{(2)}(1) q_1(s_2(F))\vac
-(r+1) q_1'(s_1(F))q_1(s_2(F))\vac
+ q_1(s_2(F))q_1(s_2(F))\vac
\\
&-\binom{r+2}{r-1} q_1^{(3)}(1)q_1(s_1(F))\vac
+\binom{r+2}{r} q_1^{(2)}(s_1(F))q_1(s_1(F))\vac
-(r+2) q_1'(s_2(F))q_1(s_1(F))\vac
\\
&+\binom{r+3}{r-1} q_1^{(4)}(1)q_1(1)\vac
-\binom{r+3}{r} q_1^{(3)}(s_1(F))q_1(1)\vac
+\binom{r+3}{r+1} q_1^{(2)}(s_2(F))q_1(1)\vac.
\end{align*}
Let us explain in details the computation of the first term of the sum:
\begin{align*} 
q_1^{(2)}(1) q_1(s_2(F))\vac 
&= (\partial q_1'(1) - q_1'(1)\partial)q_1(s_2(F))\vac \\
&= \partial q_1'(1) q_1(s_2(F))\vac - q_1'(1)\partial q_1(s_2(F))\vac.
\end{align*}
Since $q_1(s_2(F))\vac\in H^4(X)$, we have $\partial q_1(s_2(F))\vac = 0$, so:
\begin{align*}
q_1^{(2)}(1) q_1(s_2(F))\vac 
&= \partial [q_1'(1), q_1(s_2(F))] \vac + \partial q_1(s_2(F))q_1'(1)\vac.
\end{align*}
By~\cite[Theorem~3.10]{Lehn}, we have $[q_1'(1), q_1(s_2(F))] = - q_2(s_2(F))$, and we observe that $q_1'(1)\vac = \partial q_1(1)\vac = c_1(\mathcal{O}_X) = 0$. The relation between the derived Nakajima operators and the Virasoro operators given again in ~\cite[Theorem~3.10]{Lehn} gives, using $\partial \vac = 0$:
\[
\partial q_2(s_2(F))\vac =  q_2'(s_2(F))\vac = (q_1(1)q_1(1))\delta_\ast (s_2(F))\vac = \int_X s_2(F) \cdot q_1([x])q_1([x])\vac,
\]
where $\delta_\ast\colon H^\ast(X)\to H^\ast(X)\otimes H^\ast(X)$ is the push-forward map associated to the diagonal embedding and $[x]\in H^4(X)$ is the class of a point. Using similar computations, that can be performed efficiently with the help of a computer, we get:
\begin{align*}
q_1^{(2)}(1) q_1(s_2(F))\vac &= - \int_X s_2(F) \cdot q_1([x])q_1([x])\vac\\
q_1'(s_1(F))q_1(s_2(F))\vac &= 0 \\
q_1(s_2(F))q_1(s_2(F))\vac &= \left( \int_X s_2(F)\right)^2 \cdot q_1([x])q_1([x])\vac\\
q_1^{(3)}(1)q_1(s_1(F))\vac &= 0\\
q_1^{(2)}(s_1(F))q_1(s_1(F))\vac &= - \int_X s_1(F)^2 \cdot q_1([x])q_1([x])\vac\\
q_1'(s_2(F)q_1(s_1(F))\vac &=  0 \\
q_1^{(4)}(1)q_1(1)\vac &= 24\cdot q_1([x])q_1([x])\vac\\
q_1^{(3)}(s_1(F))q_1(1)\vac &= 0\\
q_1^{(2)}(s_2(F))q_1(1)\vac &= -\int_X s_2(F) \cdot q_1([x])q_1([x])\vac.
\end{align*}
Since $\int_{X^{[2]}} q_1([x])q_1([x])\vac = 1$, the value of $ \int_X s_4(F^{[2]})    $ follows directly. Similar computations give the formula for $ \int_X s_6(F^{[3]})    $. \end{proof}

\begin{remark}\label{rem:oprea}
The recursive formula in Proposition \ref{prop:segre_taut} allows us to find all Segre classes for tautological bundle $F^{[k]}$. If we restrict to top ones, there is also 
another enumerative approach to compute them which was pointed out to us by Dragos Oprea. 
 As an example, a closed formula for $k=2$ can be found in~\cite[Remark~2]{MOP}, which in fact coincides with that in Corollary~\ref{cor:segre_taut} above, but the calculations are very different. In 
Corollary~\ref{cor:segre_taut} we give a closed formula for top Segre classes for $k=3$, which is not explicitly written in~\cite{MOP}. 
\end{remark}

\subsection{Global generation of tautological bundles} \label{ss:global_gen_taut} This section will be focused on finding sufficient 
conditions for global generation of tautological vector bundles on $X^{[k]}$, for any integer $k \geqslant 2$. To do so, we first recall some useful terminology 
(cf.\;e.g.\;\cite{Knutsen} for line bundles and \cite{TT} for the more general set--up of vector bundles).

Let $X$ be any smooth, irreducible projective surface and let $G$ be a rank~$r$ vector bundle on $X$, for $r \geqslant 1$. Let $k \geqslant 1$ be an integer. 
One says that $G$ is {\em $(k-1)$--very ample} on $X$ if, for any $0$--dimensional subscheme $\xi$ of $X$ of length $h^0(\xi, {\mathcal O}_{\xi}) = k$, 
the natural evaluation map 
$$H^0(X, G) \stackrel{ev_{G, \xi}}{\longrightarrow} H^0(X, G \otimes {\mathcal O}_{\xi})$$is surjective. Notice that $G$ is $0$-very ample if and only if it is globally generated.

In this set--up, we prove the following result which will be used later on. 

\begin{proposition}\label{prop:speriamo} Let $X$ be a smooth, irreducible projective surface and let $G$ be a rank  $r$ vector bundle on $X$, with $r \geqslant 1$. Let $k \geqslant 2$ be any integer.  
\begin{itemize}
\item[(i)]\label{prop:speriamo_i} $G$ is $(k-1)$-very ample on $X$ if and only if the tautological bundle $G^{[k]}$ is globally generated (i.e. $0$--very ample) on $X^{[k]}$.

\item[(ii)]
 If $G$ is globally generated on $X$ and if $L$ is a $(k-1)$--very ample line bundle on $X$, then $G \otimes L$ is $(k-1)$--very ample on $X$, equivalently 
$(G \otimes L)^{[k]}$ is globally generated (i.e. $0$--very ample) on $X^{[k]}$. 
\end{itemize}
\end{proposition}

\begin{proof}\text{}
	
\noindent(i) $(\Rightarrow)$ If $G$ is $(k-1)$--very ample, then for any $0$--dimensional subscheme $\xi$ of $X$ of length $h^0({\mathcal O}_{\xi}) = k$, 
the evaluation map $$H^0(X, G) \stackrel{ev_{G, \xi}}{\longrightarrow} H^0(X, G \otimes {\mathcal O}_{\xi})$$is surjective. By  \eqref{eq:Fkbis}, we have that $H^0(X, G \otimes {\mathcal O}_{\xi}) \cong (G^{[k]})(\xi)$. Moreover, by \eqref{eq:Krugiso}, one has $H^0(X^{[k]}, G^{[k]}) \cong H^0(X, G)$ therefore the surjectivity of $ev_{G, \xi}$, for any $\xi \in X^{[k]}$, 
implies that $G^{[k]}$ is globally generated on $X^{[k]}$. 

\noindent
($\Leftarrow$) Conversely, assume that $G^{[k]}$ is globally generated on $X^{[k]}$. Thus, there exists a surjective map 
$$H^0(X^{[k]}, G^{[k]}) \otimes {\mathcal O}_{X^{[k]}} \rightarrow G^{[k]} \to 0 $$ so, for every $\xi \in X^{[k]}$, we have a surjective map  
$$H^0(X^{[k]}, G^{[k]}) \otimes {\mathcal O}_{X^{[k]}, \xi} \rightarrow G^{[k]}(\xi) \to 0.$$As above, by \eqref{eq:Krugiso}, we have  
$H^0(X^{[k]}, G^{[k]}) \simeq H^0(X,G)$ moreover, by \eqref{eq:Fkbis},  we have $G^{[k]}(\xi) \simeq H^0(X, G \otimes {\mathcal O}_{\xi})$. 
This implies that, for every $\xi \in X^{[k]}$, the exact sequence $0 \to {\mathcal I}_{\xi/X} \otimes G \to G \to G \otimes {\mathcal O}_{\xi} \to 0$ on $X$ gives rise in cohomology to the 
surjective map $H^0(X, G) \to  H^0(X, G \otimes {\mathcal O}_{\xi}) \to 0$, i.e., $G$ is $(k-1)$-very ample on $X$.

\vskip 2pt

\noindent
(ii) The proof is inspired by that in \cite[Lemma\;2.2]{BS}. Let $\xi$ be any $0$-dimensional subscheme of $X$ of length 
$h^0({\mathcal O}_{\xi}) = k$; the  $(k-1)$-very ampleness of $L$ ensures that, for any such $\xi$, 
the evaluation map $$H^0(X, L) \stackrel{ev_{L, \xi}}{\longrightarrow} H^0(X, L \otimes {\mathcal O}_{\xi})$$is surjective. 
 
Let ${\rm Supp} (\xi) = \{ x_1, x_2, \ldots, x_s\}$ be the support of $\xi$, where $1 \leqslant s \leqslant k$ is an integer; notice that 
the equality $s=k$ holds if and only if $\xi$ consists of $k$ distinct reduced points of $X$ whereas $s=1$ holds if and only if 
$\xi$ is a $0$-dimensional subscheme of $X$ concentrated at just one point and of length $k$. The proof below is given for $s \geq 2$. The case $s=1$ can be dealt with similarly, and it will not be specified any further.

For any integer $1 \leqslant i \leqslant s$, denote by $\xi_i \subset \xi$ the maximal subscheme of $\xi$ whose support is the point $x_i$, i.e. 
denoting by $\eta_i:= \xi \setminus \xi_i$ the {\em residual} $0$-dimensional subscheme of $\xi_i$ in $\xi$, one has 
${\rm Supp}(\eta_i) = \{x_1, x_2, \ldots , x_{i-1}, x_{i+1}, \ldots, x_s\}$.  Set $k_i:= h^0({\mathcal O}_{\xi_i}) = {\rm length(\xi_i)} \geqslant 1$, so that 
$\sum_{i=1}^s k_i = k$ and $h^0({\mathcal O}_{\eta_i}) = {\rm length(\eta_i)} = k - k_i$.  

Consider the vector subspace
$$V_{i} := H^0(X, L \otimes {\mathcal I}_{\eta_i/X}) \subset H^0(X, L),$$where ${\mathcal I}_{\eta_i/X} \subset {\mathcal O}_X$ denotes the ideal sheaf of $\eta_i$ in $X$, 
$1 \leqslant i \leqslant s$. The $(k-1)$--ampleness of $L$ ensures that $V_i \subsetneq H^0(X, L)$, that $V_i \neq V_j$ for $1 \leqslant i \neq j \leqslant s$ and, moreover, that 
the restriction of $ev_{L, \xi}$ to $V_i$ induces a surjective map 
\[ 
V_i \stackrel{\alpha_i}{\twoheadrightarrow} H^0(X, L \otimes {\mathcal I}_{\eta_i/X} \otimes {\mathcal O}_{\xi_i}) \cong H^0(L \otimes {\mathcal O}_{\xi_i}) \cong {\mathbb C}^{k_i}_{(x_i)},
\] $1 \leqslant i \leqslant s$, where the first isomorphism on the right follows from the fact that ${\rm Supp}(\eta_i) \cap {\rm Supp} (\xi_i) = \emptyset$ whereas the second isomorphism  holds by the definition of $k_i = {\rm length(\xi_i)}$ and $x_i = {\rm Supp} \{\xi_i\}$. Here $ {\mathbb C}^{k_i}_{(x_i)}$ is the stalk at the point~$x_i$.

Similarly, global generation of $G$ ensures that, for any $1 \leqslant i \leqslant s$, the evaluation map  
$$H^0(X, G) \stackrel{ev_{G, x_i}}{\longrightarrow} H^0(X, G \otimes {\mathcal O}_{x_i})$$is surjective. To ease notation,
set $\beta_i := ev_{G, x_i}$, so that we have epimorphisms 
\[ 
H^0(X, G) \stackrel{\beta_i}{\twoheadrightarrow} H^0(X, G \otimes {\mathcal O}_{x_i}) \cong {\mathbb C}^{r}_{(x_i)},
\] $1 \leqslant i \leqslant s$, where the isomorphism on the right follows from the fact that $G$ has rank $r$. 

Consider $H^0(X, G) \otimes V_i \subset H^0(X,G) \otimes H^0(X,L)$, $1 \leqslant i \leqslant s$. If we denote by 
$\mu:= \mu_{G,L}$ the natural multiplication map among global sections 
$H^0(X, G) \otimes H^0(X, L) \stackrel{\mu}{\longrightarrow} H^0(X, G \otimes L)$, 
set $$\mathcal E_i := \mu (H^0(X, G) \otimes V_i) \subset H^0(X, G \otimes L).$$By definition of $V_i$, notice that 
$\mathcal E_i \subseteq H^0(X, G \otimes L \otimes {\mathcal I}_{\eta_i/X})$. We consider the evaluation map $$ev_{G \otimes L, \xi} : H^0(X, G \otimes L) \longrightarrow H^0(X, G \otimes L \otimes {\mathcal O}_{\xi}),$$ and we set 
$\rho_i := {ev_{G \otimes L, \xi}}_{|_{\mathcal E_i}}$, $1 \leqslant i \leqslant s$. By definition of $V_i$, one has that 
$$\mathcal E_i \stackrel{\rho_i}{\longrightarrow} H^0(G \otimes L \otimes {\mathcal O}_{\xi_i}), \;\;\; 1 \leqslant i \leqslant s.$$

\begin{claim}\label{cl:surj} The map $\rho_i$ is surjective, for any $1 \leqslant i \leqslant s$. 
\end{claim}
\begin{proof}[Proof of Claim \ref{cl:surj}] For any $1 \leqslant i \leqslant s$ we have a commutative diagram 
\[
\begin{array}{ccc}
H^0(X,G) \otimes V_i & \stackrel{\beta_i \otimes \alpha_i}{\longrightarrow} & {\mathbb C}^{r}_{(x_i)} \otimes  {\mathbb C}^{k_i}_{(x_i)} \\
\downarrow^{\mu_i} & & \downarrow^{\varphi_i} \\
\mathcal E_i & \stackrel{\rho_i}{\longrightarrow} & {\mathbb C}^{k_i\,r}_{(x_i)}
\end{array}
\]where $\mu_i := \mu_{|_{H^0(X,G) \otimes V_i}}$ is surjective, by the definition of $\mathcal E_i$, and $\varphi_i$ is an isomorphism. Since $\varphi_i \circ (\beta_i \otimes \alpha_i) = \rho_i \circ \mu_i$, to prove the surjectivity of $\rho_i$ it suffices to show that 
$\beta_i \otimes \alpha_i$ is surjective. In turn, this directly follows from the fact that both $\alpha_i$ and $\beta_i$ are surjective, as observed above, 
and from the properties of tensor product (or {\em Kronecker product}) of linear maps, i.e. ${\rk}(\beta_i \otimes \alpha_i) = {\rk} (\beta_i) {\rk} 
(\alpha_i)$ (cf.\;e.g.\cite[Ex.\;4.2.1]{HJ}). Therefore $\rho_i$ is surjective, for any $1 \leqslant i \leqslant s$. 
\end{proof} 

By definition of $\mathcal E_i$ and by Claim \ref{cl:surj}, global sections in $\mathcal E_i$ vanish at $\eta_i$ but generate the stalk 
$G \otimes L \otimes {\mathcal O}_{\xi_i}$, $1 \leqslant i \leqslant s$. Since $\mathcal E_i \subset H^0(X, G \otimes L)$, for any $1 \leqslant i \leqslant s$, 
and since $\rho_i = {ev_{G \otimes L, \xi}}_{|_{\mathcal E_i}}$, this implies that global sections of $H^0(X, G \otimes L)$ separate the scheme $\xi$ via the evaluation 
map $ev_{G \otimes L, \xi}$, i.e. $ev_{G \otimes L, \xi}$ is surjective. Since $\xi$ is arbitrary, previous arguments imply that $G \otimes L$ is $(k-1)$-very ample.  The last part of (ii) directly follows from part (i). 
\end{proof}

\begin{remark} Given $F$ a rank $r$ vector bundle on $X$, \cite[Proposition 2]{MOP} gives sufficient numerical conditions on $\langle v(F), v(F)\rangle$  so that $F$ turns out to be 
$(k-1)$-very ample. These formulas however do not apply to our cases since the bundles that we study in \S\ref{bundlesHK} below are such that $c_1(F)$ is a multiple of the polarization. Besides, Proposition~\ref{prop:speriamo} gives sufficient geometric conditions to get $(k-1)$-very ampleness for a bundle $F = G \otimes L$ and so 
global generation of $F^{[k]}$, with no assumptions on $v(F)$.
\end{remark}

%%%%%%%%%%%%%%%%%%%%%%%%%%%%%%%%%%%%%%%%%%
%%%%%%%%%%%%%%%%%%%%%%%%%%%%%%%%%%%%%%%%%%
%%%%%%%%%%%%%%%%%%%%%%%%%%%%%%%%%%%%%%%%%%

\section{On some big and stable tautological bundles on $Hilb^k$ of a $K3$ surface}\label{bundlesHK} In this section we extend results proved 
in \S\ref{bundlesK3} for $K3$'s to Hyper--K\"ahler varieties given by $X^{[k]} \coloneqq Hilb^k(X)$, the {\em Hilbert scheme} parametrizing $0$--dimensional 
subschemes of $X$ of length $k$, where $k \geq 2$ is an integer and where 
$(X,H)$ is a very--general, primitively polarized $K3$ surface. 
To ease notation, we will set $Y\coloneqq X^{[k]}$  and $H_Y \coloneqq D_k(H)$, according to \eqref{eq:Dk}. 
As already observed, since $H$ is globally generated, ample and effective, then $H_Y$ is a big and nef line bundle on $Y$. Moreover, for any locally free sheaf $\mathcal F$ on 
$Y$ and any positive integer $n$ we will simply set$${\mathcal F} (n) \coloneqq {\mathcal F} \otimes H_Y^{\otimes n}.$$

Our discussion unfolds in a fashion that parallels the examples in \S\ref{bundlesK3}. Therefore, we will start with the tangent bundles.

%%%%%%%%%%%%%%%%%%%%%%%%%%%%%%%%%%%%%%%%%%
%%%%%%%%%%%%%%%%%%%%%%%%%%%%%%%%%%%%%%%%%%
%%%%%%%%%%%%%%%%%%%%%%%%%%%%%%%%%%%%%%%%%%

\subsection{The tangent bundle of $Hilb^k$ of a $K3$}\label{tangentHilb} From \S\,\ref{tangent}, when $(X,H)$ is a very--general $K3$ surface with polarization of genus $g \geq 2$, Theorem~\ref{thm:tangK3} gives sufficient conditions for $T_X(n)$ to be big and $\mu_H$-stable. We will make use  of Theorem~\ref{thm:tangK3} to prove the main result of this section, namely Theorem \ref{thm:tangHilbK3} below.

\begin{theorem}\label{thm:tangHilbK3} Let $(X, H)$ be a very--general $K3$ surface of genus $g \geq 2$. Let $Y= X^{[k]}$ be the Hilbert scheme parametrizing $0$--dimensional subschemes of $X$ of length $k$ and let $H_Y= D_k(H)$ be the big and nef line bundle defined as in \eqref{eq:Dk}. Let $T_Y$ denote the tangent bundle of $Y$. 
Then,  the vector bundles $(T_X)^{[k]}(n)$ (cf.\,\eqref{eq:Fk}) and $T_Y(n)$ are $\mu_{H_Y}$-stable, of rank $2k$ on $Y$, for any integer $n$. 
Furthermore, $(T_X)^{[k]} (n)$ and $T_Y(n)$ are also big if: 

\begin{itemize}
\item[(1)] $n \geq 5$, for $g =2$
\item[(2)] $n \geq 4$, for $g=3$
\item[(3)] $n \geq 3$, for $4 \leq g \leq 9$ or $g=11$
\item[(4)] $n \geq 2$, for $ g \geq 10$ but $g \neq 11$.
\end{itemize} 
\end{theorem}
\begin{proof} From \cite[Theorem B]{Sta}, one has an exact sequence 
$$0 \to (T_X)^{[k]} \to T_Y \to Q \to 0,$$where $Q$ is a torsion sheaf on $Y$ supported on the exceptional divisor ${\mathcal E}$; this implies that 
$ (T_X)^{[k]}$ and $T_Y$ are vector bundles of the same rank $2k$ on $Y$. The same conclusion holds for $ (T_X)^{[k]}(n)$ and $T_Y(n)$, for any integer $n$.

Focusing on $\mu_{H_Y}$--stability, from \eqref{eq:tangproperties}--(iii) we know that $T_X$ is $\mu_H$--stable on $X$ so, from  Theorem \ref{thm:Stapleton}, 
$(T_X)^{[k]}$ is $\mu_{H_Y}$-stable on $Y$. Since slope--stability is preserved under tensor product via line--bundles, one deduces that $(T_X)^{[k]}(n)$ 
is $\mu_{H_Y}$-stable, for any integer $n$. 

The following $\mu_{H_Y}$-stability argument has been communicated to us by Dragos Oprea. Let $W\subset T_Y$ be a subsheaf with $0<\rk W<\rk T_Y$. Since $\deg(T_Y) = c_1(T_Y) \cdot H_Y^{2k-1} = 0$, we have to show that $c_1(W) \cdot H_Y^{2k-1} < 0$. Define subsheafs $A$ and $B$ completing the commutative diagram with exact rows:
\[
\xymatrix{0 \ar[r] & A \ar[r] \ar@{^(->}[d] & W \ar[r] \ar@{^(->}[d] & B \ar[r] \ar@{^(->}[d] & 0 \\
	0\ar[r]  & T_X^{[k]}\ar[r]  & T_Y \ar[r] & Q \ar[r]  & 0}
\]
Since $A$ is nonzero (otherwise $W$ would be a torsion sheaf) and $T_X^{[k]}$ is $H_Y$-stable, we have $c_1(A)
\cdot H_Y^{2k-1} < 0$. Since $Q$~is supported on $\mathcal{E}$, $c_1(B)$ is a multiple of $\mathcal{E}$ hence $c_1(B) \cdot H_Y^{2k-1}=  0$, the result follows so we conclude that $T_Y$ is $\mu_{H_Y}$--stable. As above, since slope--stability is preserved under tensor product via line--bundles, one deduces that $T_Y(n)$ is $\mu_{H_Y}$-stable, for any integer $n$.

The rest of the proof will be devoted to the  ``bigness part" of the statement. From Lemma \ref{lem:effectivness}, we know that $T_X(n_0(g)) \coloneqq T_X \otimes H^{\otimes n_0(g)}$ is an effective vector bundle, the integer $n_0(g)$ depending on $g$ defined as follows:   

\[
\begin{array}{|c|c|c|c|c|c|c|c|c|c|c|c|}
\hline
g & 2 & 3 & 4 & 5 & 6 & 7 & 8 & 9  & 10 & 11 & \geq 12\\\hline
%\left\lceil \sqrt{\frac{10}{g-1}} \;\;\right\rceil & 3,16 & 2,23 & 1,82 & 1,58 & 1,41 & 1,29 & %1,19 & 1,11\\\hline
n_0(g) & 4 & 3 & 2 & 2 & 2 & 2 & 2 & 2  & 1 & 2 & 1\\\hline 
\end{array}
\]
 
Applying \eqref{eq:Krugiso} with $F = T_X$ and $L = H^{\otimes n_0(g)} = n_0(g) H$ (recall we interchangeably identify divisors and line bundles and use additive notation for divisor 
equivalently to tensor products of line bundles) one has 
$$H^0(X^{[k]}, (T_X)^{[k]} \otimes D_k(n_0(g) H) \cong H^0(X, T_X (n_0(g))) \otimes S^{k-1} \left( H^0(X,  n_0(g) H ) \right)$$ which shows that the vector bundle 
$$E\coloneqq (T_X)^{[k]} \otimes D_k(n_0(g) H)) = (T_X)^{[k]} (n_0(g) H_Y)$$ is an effective vector bundle on $Y$. 

Considering the projective bundle 
$\mathbb{P}(E) \stackrel{\pi}{\longrightarrow} Y$, then $\xi\coloneqq c_1({\mathcal O}_{\mathbb{P}(E) }(1))$ is an effective line bundle on $Y$. 
Taking into account that $H_Y$ is big and nef (since $H$ is very--ample on $X$) then, from \cite[Corollary\,2.2.7, p.\,141]{Laz1}, it follows that for any ample line bundle $A_Y$ on $Y$ 
there exist a positive integer $m_Y\coloneqq m_{A_Y}$ and an effective line bundle $N_Y \coloneqq N_{A_Y}$ such that 
\begin{equation}\label{eq1}
	m_Y H_Y \sim A_Y + N_Y,
\end{equation}
where $\sim$ denote linear equivalence of divisors on $Y$. On the other hand, since $\xi$ is $\pi$-ample then, from \cite[Proposition 1.7.10, p. 97]{Laz1}, it follows that 
\begin{equation}\label{eq2}
\xi + \pi^*(m A_Y)
\end{equation}
is ample, for any integer $m \gg 0$. Notice that, from~\eqref{eq1} above, one has that for any integer $m \gg 0$
$$m (m_Y H_Y) \sim m A_Y + m N_Y.$$Therefore, using~\eqref{eq2}, for any $m \gg 0$ one has 
$$m\,m_Y (\xi + \pi^*(H_Y)) = m\,m_Y \xi + m \pi^*(m_Y H_Y) \sim m\,m_Y \xi + \pi^*(m A_Y) + \pi^*(m N_Y).$$
Notice that 
$$m\,m_Y \xi + \pi^*(m A_Y) + \pi^*(m N_Y) = (\xi + \pi^*(m A_Y)) + \left( (m\,m_Y -1) \xi + \pi^*(m N_Y) \right)$$where the first summand on the right-side of the equality is ample 
by~\eqref{eq2} whereas the second summand is effective since $m\,m_Y -1 > 0$ and since $\xi$ and $\pi^*(m N_Y)$ are both effective. Thus, from \cite[Corollary 2.2.7 (iii), p. 141]{Laz1}, it follows that 
$\xi + \pi^*(H_Y)$ is a big line bundle on $\mathbb{P}(E)$ and so 
$$E \otimes H_Y = \left((T_X)^{[k]} \otimes (n_0(g) H_Y)\right) \otimes H_Y = (T_X)^{[k]} \otimes {\mathcal O}_Y ((n_0(g) +1) H_Y)$$is a big vector bundle on $Y$. 
Since $H_Y$ is big and nef, then $(T_X)^{[k]} \otimes {\mathcal O}_Y (n H_Y)$ is a big vector bundle, for any $n \geq n_0(g) +1$.

Finally, consider the exact sequence 
$$ 0 \to (T_X)^{[k]} \to T_Y \to Q \to 0$$from \cite[Theorem B]{Sta}, where $ (T_X)^{[k]}$ and $T_Y$ are vector bundles on $Y$ of the same rank $2k$ whereas $Q$ 
is a torsion sheaf on $Y$. Using the fact that $(T_X)^{[k]} \otimes {\mathcal O}_Y (n H_Y)$ is big for any $n \geq n_0(g) +1$, the exact sequence above, the fact that 
$(T_X)^{[k]}$ and $T_Y$ have the same rank and finally the characterization of bigness in terms of global sections of the corresponding tautological divisors 
on $\mathbb{P}(E)$ as in \cite[Lemma 2.2.3, p. 139]{Laz1}, it follows that $T_Y(n H_Y)$ is big for any $n \geqslant n_0(g) +1$, which completes the proof of the statement. 
\end{proof}

\begin{remark}\label{rm:stability}
	Since $Y$ is an irreducible holomorphic symplectic manifold, it satisfies in particular assumptions as in \cite[Definition\;(8.16.2)]{GKP}: indeed, in \cite[beginning of\;\S2]{HP} it is observed that when $Y$ is smooth 
	(as it occurs in our case), by the purity of the branch locus, any {\em quasi--\'etale} morphism $f: Y' \to Y$ (i.e. $f$ \'etale in codimension one, using same terminology as in \cite[Definition\;(8.16.2)]{GKP}) is actually \'etale. On the other hand, since $Y$ is simply connected, any \'etale $f:Y' \to Y$ is actually an isomorphism. Thus, the global generation assumption on exterior algebra of forms is 
	satisfied. One can therefore apply \cite[Prop.\;8.20]{GKP} to get that $T_Y$ is {\em strongly stable} (in the sense of \cite[Def.\,7.2]{GKP}). This implies in particular that 
	$T_Y$ is $\mu_A$-stable w.r.t. any ample line bundle $A \in {\rm Pic}(Y)$. Being $\mu_A$--stable for any (ample) polarization $A$, then in particular $T_Y$ is {\em simple}, i.e. 
	$\End(T_Y) \cong \mathbb{C}$ (cf.\;\cite[Corollary\,1.2.8]{HL}).
\end{remark}

%%%%%%%%%%%%%%%%%%%%%%%%%%%%%%%%%%%%%%%%%%%%%%%%%%
%%%%%%%%%%%%%%%%%%%%%%%%%%%%%%%%%%%%%%%%%%%%%%%%%%
%%%%%%%%%%%%%%%%%%%%%%%%%%%%%%%%%%%%%%%%%%%%%%%%%%
%%%%%%%%%%%%%%%%%%%%%%%%%%%%%%%%%%%%%%%%%%%%%%%%%%

\subsection{Big and stable tautological bundles on $X^{[k]}$ arising from line bundles on $X$ a very--general $K3$}\label{tautHilblb} 
Examples of further tautological bundles on $Y= X^{[k]}$, which are big and $\mu_{D_k(H)}$--stable, when $X$ is a very--general primitively polarized $K3$ surface of genus $g$ can be easily 
obtained as follows.  

\begin{theorem}\label{thm:lbHilb} Let $k \geqslant 2$ be any integer and let $(X,H)$ be a very--general primitively 
polarized $K3$ surface of genus $g > 2k -2$. Let $L_n := H^{\otimes n} \in {\rm Pic}(X)$, where $n \geqslant 1$ any integer. Set $Y:= X^{[k]}$.  
Then, the rank~$k$ tautological vector bundle $(L_n)^{[k]}$ on $Y$ is globally generated and $\mu_{D_k(H)}$--stable. 
If moreover $   \int_Y s_{2k} ((L_n)^{[k]})    >0$ then $(L_n)^{[k]}$ is also a big vector bundle on $X^{[k]}$.   
\end{theorem} 

\begin{proof} For $n=1$, i.e. $L_1 = H$, C. Voisin~\cite[Lemma~2.2]{V} proves that $H^{[k]}$ is generated by global sections when $g > 2 k - 2$. 
If $n >1$, we use \cite[Theorem~1.1]{Knutsen}, which gives necessary and sufficient conditions for the line bundle $L_n$ to be $(k-1)$--very ample. Indeed 
it is a straightforward computation to show that, if $k \geqslant 2$ and $g > 2 k-2$, then $L_n^2 \geqslant 4 (k-1)$  holds and moreover that 
there are no effective divisors $D$ on $X$ such that$$2 D^2 \leqslant L_n \cdot D \leqslant D^2 + k \leqslant 2k.$$Therefore, since condition (iii) in \cite[Theorem~1.1]{Knutsen} holds true, it follows that under the numerical assumptions $k \geqslant 2$ and  $g > 2k-2$, $L_n$ is $(k-1)$--very ample for any $n \geqslant 1$. 
Thus, from Proposition \ref{prop:speriamo}--(i), it follows that $(L_n)^{[k]}$ is globally generated on $X^{[k]}$.

Since any line bundle is $\mu_H$--stable on $X$, then the rank~$k$ vector bundle $L_n^{[k]}$ is certainly $\mu_{D_k(H)}$-stable on $Y$, for any $k\geqslant 2$ and any $n \geqslant 1$, 
as it follows from Theorem~\ref{thm:Stapleton}.

Finally, since $k \geqslant 2$ and $g > 2k-2$ imply that $(L_n)^{[k]}$ is globally generated on $X^{[k]}$, from Proposition \ref{prop:bigness} we know that 
$(-1)^{2k} \int_Y s_{2k}( (L_n)^{[k]} ) =  \int_Y s_{2k}( (L_n)^{[k]} )    >0$ implies that $(L_n)^{[k]}$ is big. \end{proof}

As a direct consequence of the previous result we have the following:

\begin{corollary}\label{cor:lbHilb} 
Let $k \in \{2,\;3\}$ be an integer and let $(X,H)$ be a very--general primitively 
	polarized $K3$ surface of genus $g > 2k -2$. Let $L_n := H^{\otimes n} \in {\rm Pic}(X)$, for any integer $n \geqslant 1$. Set $Y:= X^{[k]}$. 
	Then, the rank~$k$ tautological vector bundle 
	$(L_n)^{[k]}$ is globally generated and $\mu_{D_k(H)}$--stable on $Y$. If moreover one has $n \geqslant 2$, 
	then $(L_n)^{[k]}$ is also big.
\end{corollary}

\begin{proof} Since $g > 2k-2$, from Theorem \ref{thm:lbHilb} one immediately deduces global generation and $\mu_{D_k(H)}$--stability of $(L_n)^{[k]}$. 
	For the rest of the statement, if $k=2$, Corollary~\ref{cor:segre_taut} gives that 
	$$   \int_Y s_4((L_n)^{[k]})    = 2 (n^4 (g-1)^2 - 5 n^2 (g-1) + 6).$$Thus, $n^2 > \frac{3}{g-1}$ implies that 
	$\int_Y s_4((L_n)^{[k]})    $ is certainly positive. 
	Since $k=2$ and $g > 2k-2 = 2$, notice that $\frac{3}{g-1} < 3$. Therefore, if $n \geqslant 2$, $   \int_Y s_4((L_n)^{[k]})     > 0$ holds true. If otherwise $k=3$, by Corollary~\ref{cor:segre_taut} we get 
	$$ 
	\int_Y s_6((L_n)^{[k]})     = \frac{1}{3} \left( 4 n^6 (g-1)^3 + 3 n^4 (g-1)^2 + 684 n^2 (g-1) - 480 \right).
	$$ 
	This equals $\frac 43 (g-1)^3 \left(n^2 - \frac{4}{g-1} \right)\left( n^2 - \frac{5}{g-1}\right) \left( n^2 - \frac{6}{g-1}\right)$, which is positive for $g >4$ and $n \geq 2$.  
\end{proof}

%%%%%%%%%%%%%%%%%%%%%%%%%%%%%%%%%%%%%%%%%%
%%%%%%%%%%%%%%%%%%%%%%%%%%%%%%%%%%%%%%%%%%
%%%%%%%%%%%%%%%%%%%%%%%%%%%%%%%%%%%%%%%%%%

\subsection{Big and stable bundles on $X^{[k]}$ arising from Mukai--Lazarsfeld bundles on $X$ a very--general $K3$}\label{tautHilbodd} 
Taking into account what proved in \S\;\ref{oddrank}, here we have the following: 

\begin{theorem}\label{prop:oddvectHilb} Let $k \geqslant 2$ be a positive integer. Let $(X,H)$ be a very--general polarized $K3$ surface of genus $g > 2k-2$. 
Let $(g, r-1, d)$ be any triple of positive integers such that $ d < 2g-2$ and $\rho(g,r-1,d) \geqslant 0$, where $\rho(g, r-1, d)$ the {\em Brill--Noether number} 
as in \eqref{eq:rhoA}. Set $Y:= X^{[k]}$.
Then, for any Mukai--Lazarsfeld rank~$r$ vector bundle $E$ on $X$ as in Theorem \ref{prop:oddvect}, the tautological 
rank~$kr$ vector bundle $(E\otimes H)^{[k]}$ is globally generated and $\mu_{D_k(H)}$--stable on $Y$.
If moreover $   \int_Y s_{2k} ((E\otimes H)^{[k]})    >0$, then  $(E\otimes H)^{[k]}$ is also big. 
\end{theorem}

\begin{proof} From Theorem \ref{prop:oddvect}, any Mukai--Lazarsfeld vector bundle $E = E_{C,A}$ constructed therein is $\mu_H$--stable; so it is 
$E \otimes H$. Therefore, the rank~$rk$ vector bundle $(E\otimes H)^{[k]}$ is certainly $\mu_{D_k(H)}$-stable on $Y$, for any $k \geqslant 2$, 
as it follows from Theorem~\ref{thm:Stapleton}.

Any such $E = E_{C,A}$ is also globally generated; since, by assumption, we have $g > 2k-2$ then, from \cite[Lemma~2.2]{V} or following the arguments in the 
proof of Theorem \ref{thm:lbHilb}, the line bundle $H$ is $(k-1)$--very ample on $X$. Therefore, from Proposition \ref{prop:speriamo}--(ii), 
$E \otimes H$ is $(k-1)$-very ample on $X$ so,  by Proposition \ref{prop:speriamo}--(i), $(E\otimes H)^{[k]}$ is globally generated on~$Y$.  
Thus, from Proposition \ref{prop:bigness}, $(-1)^{2k} \int_Y s_{2k} ((E\otimes H)^{[k]}) = \int_Y s_{2k} ((E\otimes H)^{[k]})    >0$ 
implies that  $(E\otimes H)^{[k]}$ is big. \end{proof}

Recall that Mukai-Lazarsfeld bundles $E = E_{C,A}$ as above are such that
$$
{\rk}(E) = r, \;\; c_1(E) = H, \;\;    \int_X c_{2} (E)    = d.$$
Therefore, from \eqref{eq:dualchern}, we have
$$\rk(E \otimes H) = r,   \;\; c_1(E\otimes H) = r H + c_1(E) = (r+1) H, \;\; \int_X c_{2} (E \otimes H)    = 2 \left(\binom{r}{2} +(r-1)\right) (g-1) +d.$$Morevover, from
\eqref{eq:dualsegre}, we have 
$$s_1 (E \otimes H) = - c_1(E \otimes H), \;\; s_2(E \otimes H) = c_1(E \otimes H)^2 - c_2(E \otimes H).$$Using these expressions, one has:

\begin{corollary}\label{cor:oddvectHilb}  Let $k \in \{2,\;3\}$ be an integer and let $(X,H)$ be a very--general primitively 
polarized $K3$ surface of genus $g > 2k -2$. Let $(g, r-1, d)$ be any triple of positive integers such that $ d < 2g-2$ and $\rho(g,r-1,d) \geqslant 0$, where 
$\rho(g, r-1, d)$ the {\em Brill--Noether number} as in \eqref{eq:rhoA}. Then the rank~$rk$ vector bundle $(E \otimes H)^{[k]}$ is globally generated, $\mu_{D_k(H)}$--stable and big on $X^{[k]}$ for $k=2,3$. 
\end{corollary} 

\begin{proof} This is a direct consequence of Theorem \ref{prop:oddvectHilb} and of Corollary~\ref{cor:segre_taut}, namely one needs to show that, for $k=2$, $ \int_Y s_4((E\otimes H)^{[k]}) > 0$ (respectively, $\int_Y s_6((E\otimes H)^{[k]})> 0$ for $k=3$). In order to prove bigness, the strategy is similar for $k=2$ and $k=3$.  Here we illustrate the case $k=2$, the other case can be dealt with analogously.
The numerical condition coming from the positivity of the Segre classes is:
\[
	\int_{X^{[2]}}s_4\left( \left(E \otimes H\right)^{[2]}\right) = \alpha_{0, 0} + \alpha_{1,0} g + \alpha_{0, 1} d + \alpha_{1, 1} g d + \alpha_{2, 0} g ^2 + \frac{1}{2} d^2 >0,
\]
\begin{align*}
	\alpha_{0, 0} = \frac{1}{2}\left(4 r^4+23 r^3+53 r^2+58 r+30\right), \qquad
	\alpha_{1, 0} =-\frac{1}{2}\left(4 r^4+23 r^3+ 59 r^2+76 r+46\right), \\
	\alpha_{0, 1} = \frac{1}{2} \left( 3 r^2 + 9 r+11\right),	\qquad
	\alpha_{1, 1} = -\left(r^2+3 r+4\right),\qquad
	\alpha_{2, 0} = \frac{1}{2}(r^2+3r+4)^2.
\end{align*}
This polynomial expression $p(d,r,g)$ in the variables $d,r,g$ has degree $2$ in $d$ and in $g$. We look at the locus $\{ (d,g): p(d,r,g)=0\}$ as a plane conic in the real plane $(g,d)$ whose coefficients depend on $r$. By direct inspection, the conic is a parabola for any value of $r$. In fact, if we perform the coordinate change 
$$
\left\{
\begin{array}{l}
g= G - (r^2+3r+4)D, \\
d=(r^2+3r+4)G+D,
\end{array}
\right.
$$
we obtain the locus in the plane $(G,D)$ which is defined by the vanishing of the  polynomial:
\[
\beta_{0, 0} + \beta_{1, 0} G +  \beta_{0, 1} D + \beta_{0, 2} D^2
\]
whose coefficients depend on $r$:
\begin{align*}
\beta_{0, 1} &= \frac{1}{2}\left(4 r^4+23 r^3+53 r^2+ 58 r+30\right),
&\beta_{1, 0} &= -\frac{1}{2}(r+2)(r+1)^3\\
\beta_{0, 1} &= \frac{1}{2}\left(4r^6+35 r^5+144 r^4+345 r^3+513 r^2+451 r+195\right), 
&\beta_{0, 2} &= \frac{1}{2} (r^4+6 r^3+17 r^2+24 r+17)^2.
\end{align*}

In order to show that the top Segre class is always positive under the assumption $d < 2g-2$, we prove that the parabola is contained in the half-plane $d >2g-2$. For these purposes, we first show that the conic takes positive values along the line $d=2g-2$.  In other words, the parabola does not intersect this line. By connectedness, either the parabola is contained in the half-plane $d < 2g-2$, or it is contained in the half-plane $d>2g-2$. By continuity of the real parameter $r$ and the canonical form of the parabola, if the parabola is contained in one of the half-planes for one value of $r$, then it is contained in the same half-plane for every value of $r$. Therefore, it suffices to check the sign at one point of the parabola for one value of $r$. We find a point for $r=3$ that is contained in the half-plane $d>2g-2$. Hence the parabola are always contained in this half-plane, so the Segre number above is always positive. Thus the claim follows.

\end{proof}

%%%%%%%%%%%%%%%%%%%%%%%%%%%%%%%%%%%%%%%%%%
%%%%%%%%%%%%%%%%%%%%%%%%%%%%%%%%%%%%%%%%%%
%%%%%%%%%%%%%%%%%%%%%%%%%%%%%%%%%%%%%%%%%%

\subsection{Big and stable bundles on $X^{[k]}$ arising from Ulrich bundles on $X$ a very general $K3$}\label{tautHilbeven} Taking into account what proved in \S\ref{evenrank}, 
here we have the following: 

\begin{theorem}\label{prop:evenvectHilb} Let $k \geqslant 2$ and $h > 2k-3$ be positive integers. Let $(X, H)$ be a very--general polarized $K3$ surface, with polarization $H$ of 
genus $g = h+1$. Set $Y:= X^{[k]}$.
For every integer $a \geqslant 1$, consider any $\mu_H$--stable Ulrich bundle $E$ of rank~$2a$ on $X$ as in Theorem~\ref{thm:ultrich}. 
Thus, the tautological vector bundle $(E\otimes H)^{[k]}$ of rank $2ka$  is globally generated and $\mu_{D_k(H)}$--stable on $Y$. 
If moreover $ \int_Y s_{2k} (E\otimes H)^{[k]}    >0$, then  $(E\otimes H)^{[k]}$ is also big.    
\end{theorem}

\begin{proof} From Theorem \ref{thm:ultrich}, any Ulrich bundle $E$ considered therein is $\mu_H$--stable; so it is 
$E \otimes H$. Thus, the rank~$2ak$ vector bundle $(E\otimes H)^{[k]}$ is certainly $\mu_{D_k(H)}$-stable on $Y$, for any $k \geqslant 2$, as it follows 
from Theorem~\ref{thm:Stapleton}.

Since $E$ is an Ulrich bundle on $X$, in particular it is globally generated. From the  assumption $h > 2k-3$, it follows that 
$g = h+1 > 2k-2$ therefore, from \cite[Lemma~2.2]{V} (or following the arguments in the proof of Theorem \ref{thm:lbHilb}), 
$H$ is $(k-1)$--very ample on $X$. Thus, from Proposition \ref{prop:speriamo}--(ii), $E \otimes H$ is $(k-1)$-very ample on $X$ and so, by 
Proposition \ref{prop:speriamo}--(i), the tautological bundle 
$(E\otimes H)^{[k]}$ is globally generated on $Y$. Thus, from Proposition \ref{prop:bigness}, 
$(-1)^{2k}    \int_Y s_{2k} (E\otimes H)^{[k]}    = \int_Y s_{2k} (E\otimes H)^{[k]}    >0$ 
implies that  $(E\otimes H)^{[k]}$ is big on $Y$. \end{proof}

Similarly as for Mukai--Lazarsfeld vector bundles, any Ulrich bundle $E$ as above is such that
$$
{\rk}(E) = 2a, \;\; c_1(E) = 3aH, \;\;    \int_X c_2(E)    = 9a^2 h - 4a(h-1). 
$$
Therefore, from \eqref{eq:dualchern}, we have
$${\rk}(E \otimes H) = 2a,   \;\; c_1(E\otimes H)  = 2a H + c_1(E) = 5a H, \;\; \int_X c_2(E\otimes H)    = 9a^2 h - 4a(h-1) + 8a(2a-1)h.$$Morevover, from
\eqref{eq:dualsegre}, we have 
$$s_1 (E \otimes H) = - c_1(E \otimes H), \;\; s_2(E \otimes H) = c_1(E \otimes H)^2 - c_2(E \otimes H).$$Using these expressions, one has: 
  
\begin{corollary}\label{cor:evenvectHilb}  Let $k \in \{2,\;3\}$ and $h > 2k-3$ be integers. Let $(X, H)$ be a very--general polarized $K3$ surface, with polarization $H$ of 
genus $g = h+1$. For every integer $a \geqslant 1$, consider a globally generated and $\mu_H$--stable Ulrich bundles $E$ of rank~$2a$ on $X$ 
as in Theorem \ref{thm:ultrich}. Set $Y:= X^{[k]}$. Then the rank~$2ak$ vector bundle $(E \otimes H)^{[k]}$ is globally generated, $\mu_{D_k(H)}$--stable and big on $Y$ for $k=2,3$. 
\end{corollary}

\begin{proof} This is a direct consequence of Theorem \ref{prop:evenvectHilb} and the numerical conditions on the top Segre classes. For (i), the positivity of the Segre number $ \int_{X^{[4]}} s_4((E\otimes H)^{[2]})  > 0$ translates, after simplification, into the following inequality:
$$
 (60 + 521 a+ 1212 a^2+ 841 a^3)
- (36 + 581 a +  1748 a^2 + 1450 a^3) g
+ a (25 a+12)^2 g^2 >0.
$$
The left--hand--side member of the previous inequality is a degree~$2$ polynomial in the indeterminate $g$, with coefficients depending on $a$. An elementary numerical study of the real maximal root of this polynomial - as a function of $a$ - shows that the maximal root is always smaller than $2$, under the assumption $g >2$. As for (ii), the condition  $ \int_{X^{[6]}}  s_6((E\otimes H)^{[3]})    > 0$ is equivalent to the positivity of a degree--$3$ polynomial in $g$, whose coefficients depend on~$a$, namely: 
	\begin{align*}
     \int_{X^{[6]}}  s_6((E\otimes H)^{[3]}) &= \alpha_0(a) + \alpha_1(a) g + \alpha_2(a) g^2 + \alpha_3(a) g^3\\
    \alpha_0(a)&= -\frac{1}{2}(11979 a^5+28116 a^4+25173 a^3+10678 a^2+2132 a+160) a\\	
	\alpha_1(a)&= \frac{1}{6} (81675 a^5+167652 a^4+126918 a^3+42834 a^2+6020 a+240) a\\
	\alpha_2(a)&= -\frac{1}{2} (25 a+12) (825 a^3+1048 a^2+389 a+36) a^2\\
	\alpha_3(a)&= \frac{1}{6} (25 a+12)^3 a^3
	\end{align*}	
	Similarly as above, a numerical study of its maximal root, using the Cardan--Tartaglia formula shows that the maximal root of this polynomial is always smaller than $2$ under our assumptions, hence the result.
\end{proof}

\begin{remark}
Dragos Oprea informed us that the general statements given in Theorems~\ref{thm:lbHilb}, \ref{prop:oddvectHilb} \& \ref{prop:evenvectHilb} and generating series of Segre integrals of Marian--Oprea--Pandharipande~\cite{MOP}  can be  used to extend the numerical computations in Corollaries~\ref{cor:lbHilb}, \ref{cor:oddvectHilb} \& \ref{cor:evenvectHilb} to $k\geq 4$ 
and to obtain similar positivity results on the Hilbert schemes of points of abelian, bielliptic or Enriques surfaces. These computations appear in~\cite{Oprea}.
\end{remark}

Data sharing is not applicable to this article as no datasets were generated or analysed during the current study. 
On behalf of all authors, the corresponding author states that there is no conflict of interest. 

%%%%%%%%%%%%%%%%%%%%%%%%%%%%%%%%%%%%%%%%%%%%%%%%%%%%%%%%%%%%%%%%
%%%%%%%%%%%%%%%%%%%%%%%%%%%%%%%%%%%%%%%%%%%%%%%%%%%%%%%%%%%%%%%
%%%%%%%%%%%%%%%%%%%%%%%%%%%%%%%%%%%%%%%%%%%%%%%%%%%%%%%%%%%%%%%%


\begin{thebibliography}{99}
	

	%%%%%%%%%%%%%%%%%%
	
	\bibitem{AFO} M. Aprodu, G. Farkas, A. Ortega, Minimal resolutions, Chow forms and Ulrich bundles on $K3$ surfaces, {\em J. reine angew. Math.}, {\bf 730} (2017), 225--249
	
	\bibitem {B} A. Beauville, {\it Ulrich bundles on abelian surfaces}, Proc. Amer. Math. Soc. 144 (2016), 4609 -- 4611
	
	
	\bibitem{BPV} W. Barth, K. Hulek, C. Peters, A. Van de Ven, {\em Compact Complex Surfaces}, 2nd
	edn. Springer, Berlin (2004)
	
	\bibitem {BS_zero} M. Beltrametti, A. Sommese, {\it Zero cycles and kth order embeddings of smooth projective surfaces},  Sympos. Math. XXXII, Problems in the theory of surfaces and their classification, 33 -- 48, Academic Press, London, 1991
	
	\bibitem{MU} T. Bauer, S. Kov\`acs, A. K\"uronya, E. C. Mistretta, T. Szemberg, S. Urbinati, On positivity and base loci of vector bundles, {\em Europ. Journal of Math.}, {\bf 1} (2015), 229--249.
	
	\bibitem{Be} A. Beauville, Fano threefolds and $K3$ surfaces, in {\em Proceedings of the Fano Conference}, Turin 2002, Edited by A. Collino, A. Conte, M. Marchisio, pp. 175--184, 2004. 
	
	\bibitem{BB}A. Beauville, Vari\'et\'es K\"akleriennes dont la premi\`ere class de Chern est nulle, {\em J. Diff. Geom.}, {\bf 18} (1983), 755--782. 
	
	\bibitem{BS} Beltrametti M.C., Sommese A.J., On k-Jet Ampleness, in {\em Complex Analysis and Geometry. The University Series in Mathematics}, L'aquila 1993,  Edited by 
	V. Ancona, A. Silva A., pp. 355--376, 1993, Springer, Boston, MA.
	
	\bibitem{BF} G. Bini, F. Flamini, Big vector bundles on surfaces and fourfolds, {\em Mediterranean Journal of Mathematics}, {\bf 17} (2020), n. 1, art. 17, 1--20
	
	
	\bibitem{Bos} S. Boissi\`ere, Automorphismes naturels de l'espace de Douady de points sur une surface, {\em Canad. J. Math.}, {\bf 64} (2012), n. 1, 3--23
	
	\bibitem{BCNS}
	S. Boissi\`ere, A. Cattaneo, M. Nieper-Wisskirchen, and A.  Sarti, \emph{The automorphism group of the {H}ilbert scheme of two points on
		a generic projective {K}3 surface}, K3 surfaces and their moduli, Progr.
	Math., vol. 315, Birkh\"{a}user/Springer, [Cham], 2016, pp.~1--15.
	
	
	\bibitem{MR} L. Costa, R. Mir\'o-Roig, A counterexample to a conjecture due to Douglas, Reinbacher and Yau, {\em J. Geom. Phys.} {\bf 57} (2007), 2229
	
	\bibitem{Dan1} G. Danila, Sections du fibr\'e determinant sur l'espace de modules des faisceaux semi-stables
	de rang 2 sur le plan projectif, {\em Ann. Inst. Fourier (Grenoble)}, {\bf 50} (2000), n. 5, 1323--1347
	
	
	\bibitem{Dan2} G. Danila, Sur la cohomologie d'un fibr\'e  tautologique sur le schema de Hilbert d'une surface, {\em J. Algebraic Geom}, {\bf 10} (2001), n. 2, 247--280
	
	\bibitem{Dan3} G. Danila, Sections de la puissance tensorielle du fibr\'e tautologique sur le schema de Hilbert des points d'une surface, {\em Bull. Lond. Math. Soc.}, {\bf 39} (2007), n. 2, 311--316
	
	\bibitem{Dem} J.-P. Demailly, {\em Analytic methods in algebraic geometry. Surveys of Modern Mathematics}, {\bf 1}, International Press, Somerville, MA; Higher Education Press, Beijing, 2012.
	
	\bibitem{DPS} J.-P. Demailly, T. Peternell, M. Schneider, Compact complex manifolds with numerically effective tangent bundles, {\em J. Alg. Geom.} {\bf 3} (1994), 295--345.
	
	\bibitem{EGL} G. Ellinsgrud, L. Goettsche, M. Lehn, On the Cobordism Class of the Hilbert Scheme
	of a Surface, {\em J. Algebraic Geom.} {\bf 10}, (2001), n.1, 81--100
	
	\bibitem{Fogarty}
	J.~Fogarty, \emph{Algebraic families on an algebraic surface. {II}. {T}he
		{P}icard scheme of the punctual {H}ilbert scheme}, Amer. J. Math. \textbf{95}
	(1973), 660--687.
	
	
	\bibitem{FF} F. Flamini, $\mathbb{P}^r$-scrolls arising from Brill--Noether theory and $K3$--surfaces, {\em Manuscripta Mathematica}, {\bf 132} (2010), 199--220.
	
	\bibitem{Fu} W. Fulton,  {\em Intersection theory}, Ergebnisse der Mathematik und ihrer Grenzgebiete. 3 Folge, A Series of Modern Surveys in Mathematics, Springer-Verlag, Berlin, 1998. 
	
	\bibitem{KLS}
	D.~Kaledin, M.~Lehn, and Ch. Sorger, \emph{Singular symplectic moduli spaces},
	Invent. Math. \textbf{164} (2006), no.~3, 591--614.
	
	\bibitem{Lehn} M.~Lehn, \emph{Chern classes of tautological sheaves on {H}ilbert schemes of
		points on surfaces}, Invent. Math. \textbf{136} (1999), no.~1, 157--207.
	
	\bibitem{AL}A. Lopez, R. Mu\~{n}oz, On the classification on non-big Ulrich vector bundles on surfaces and threefolds, e-print arXiv:2101.04207v2 (2021).
	
	\bibitem{GKP} D. Greb, S. Kebekus, T. Peternell, Singular spaces with trivial canonical class, in 
	{\em "Minimal Models and Extremal Rays" (Kyoto 2011)}, proceedings of a conference in honor of Shigefumi Mori's 60th birthday, 
	{\em Advanced Studies in Pure Mathematics} {\bf 70} (2016), 67--114.
	
	
	\bibitem{GG} M. Green, P. Griffiths, Positivity of vector bundles and Hodge theory, {\em arxiv:1803.07405 [math.AG] 10Oct2018}, 1--97.
	
	\bibitem{H} R. Hartshorne, {\it Algebraic geometry}, Graduate Texts in Math.  {\bf 52}, Springer-Verlag, New York, 1977.
	
	
	\bibitem{HP} A. H\"oring, T. Peternell, Algebraic integrability of foliations with numerically trivial canonical bundle,  {\em Invent. Math.},  {\bf 216} (2019), 395--419.
	
	\bibitem{HJ} R.A. Horn and C.R. Johnson. {\it Topics in Matrix Analysis}. Cambridge University Press, Cambridge, 1991.
	
	
	\bibitem{Hu} D. Huybrechts, {\it Lectures on $K3$ surfaces}, Cambridge University Press, 2017. 
	
	\bibitem{HL} D. Huybrechts, M. Lehn, {\it The geometry of moduli spaces of sheaves}, Cambridge University Press, 2010.
	
	\bibitem{Knutsen} A.~L. Knutsen, \emph{On $k^{th}$--order embeddings of $K3$ surfaces and Enriques surfaces}, Manuscripta Math. \textbf{104} (2001), no.~2, 211--237.
	
	\bibitem{Kr} A. Krug, Tensor products of tautological bundles under the Bridgeland-King-Reid-Haiman equivalence, {\em Geom. Dedicata} {\bf 172} (2014), 245-291
	
	\bibitem{Kul} S. A. Kuleshov, An existence theorem for exceptional bundles on $K3$ surfaces, {\em Math. USSR Izvestiya} {\bf 34} (1990), n.2, 373--388
	
	\bibitem{Laz0} R. Lazarsfeld, Brill-Noether-Petri without degenerations, {\em J. Differential Geom.} {\bf 23}, (1986), n.3, 299--307
	
	\bibitem{Laz1} R. Lazarsfeld, {\it Positivity in Algebraic Geometry. I}, Ergebnisse der Mathematik und ihrer Grenzgebiete. 3 Folge, A Series of Modern Surveys in Mathematics, 
	Springer-Verlag, Berlin, 2004. 
	
	\bibitem{Laz2} R. Lazarsfeld, {\it Positivity in Algebraic Geometry. II}, Ergebnisse der Mathematik und ihrer Grenzgebiete. 3 Folge, A Series of Modern Surveys in Mathematics, 
	Springer-Verlag, Berlin, 2004. 
	
	\bibitem{MOP_rank1}
	A.~Marian, D.~Oprea, and R.~Pandharipande, \emph{Segre classes and {H}ilbert
		schemes of points}, Ann. Sci. \'{E}c. Norm. Sup\'{e}r. (4) \textbf{50}
	(2017), no.~1, 239--267.
	
	
	\bibitem {MOP}
	A. Marian, D. Oprea, R. Pandharipande, {\it Higher rank Segre integrals over the Hilbert scheme of points}, JEMS, published online, \texttt{doi:10.4171/JEMS/1149}
	
	\bibitem{Oprea} D.Oprea, {\it Big and nef tautological vector bundles over the Hilbert scheme of points}, SIGMA 18 (2022), 061.
	
	\bibitem{Muk} S. Mukai, On the moduli space of bundles on K3 surfaces I, in: {\em Vector Bundles on Algebraic Varieties}, Oxford, 1987, 341--413.
	
	\bibitem{ur} U. Riess, Base divisors of big and nef line bundles on irreducible symplectic varieties, e-print arXiv:1807.05192v1 (2018).
	
	
	
	\bibitem{Sca1} L. Scala, Cohomology of the Hilbert scheme of points on a surface with values in representations
	of tautological bundles, {\em Duke Math. J.} {\bf 150} (2009), no. 2,  211--267.
	
	\bibitem{Sca2} L. Scala, Some remarks on tautological sheaves on Hilbert schemes of points on a surface, {\em Geom.
		Dedicata} {\bf 139} (2009), 313--329.
	
	\bibitem{Sta} D. Stapleton, Geometry and stability of tautological bundles on Hilbert schemes of points, {\em Algebra Number Theory}, {\bf 10} (2016), No. 6, 
	1173--1190.
	
	%\bibitem {S} T. Szemberg, {\it On the positivity of line bundles on Enriques %surfaces}, Transactions of the AMS 353 (2001), 4963 -- 4972
	
	\bibitem{TT} K. Takahashi, H. Terakawa, Generalized adjunction of $k$-very ample vector bundles on algebraic surfaces, {\em Geom. Dedicata}, {\bf 71} (1998), 309--325.
	
	\bibitem{Tsu} H. Tsuji, Stability of tangent bundles of minimal algebraic varieties, {\em Topology}, {\bf 24} (1988), No. 4, 429--442.
	
	\bibitem{V} C. Voisin, Segre classes of tautological bundles on Hilbert schemes of surfaces, {\em Algebr. Geom.} {\bf 6} (2019), no. 2, 186--195.
	
	\bibitem{Yo}
	K.~Yoshioka, \emph{Twisted stability and {F}ourier-{M}ukai transform. {I}},
	Compositio Math. \textbf{138} (2003), no.~3, 261--288.
	
	\bibitem{mo} \url{https://mathoverflow.net/questions/321052/the-symmetric-power-of-a-tensor-product}
	
\end{thebibliography}
\end{document}